\documentclass[11pt,a4paper]{article}
\usepackage[T1]{fontenc}
\usepackage[utf8]{inputenc}
\usepackage{lmodern}
\usepackage{amsmath,amssymb,amsthm,mathtools}
\usepackage[margin=28mm]{geometry}
\usepackage{microtype}
\usepackage[colorlinks=true,linkcolor=blue,citecolor=blue,urlcolor=blue]{hyperref}
\hypersetup{pdftitle={Blow-up bounds under profile approximation for superconformal wave equations},
pdfauthor={Firas Kaabi}}
\newtheorem{theorem}{Theorem}[section]
\newtheorem{proposition}[theorem]{Proposition}
\newtheorem{lemma}[theorem]{Lemma}
\newtheorem{corollary}[theorem]{Corollary}
\theoremstyle{remark}
\newtheorem{remark}[theorem]{Remark}
\numberwithin{equation}{section}
\newcommand{\R}{\mathbb R}
\newcommand{\X}{\mathcal X}
\newcommand{\D}{\mathcal D}
\newcommand{\E}{\mathfrak E}
\newcommand{\LL}{\mathcal L_\alpha}
\DeclareMathOperator{\diver}{div}

\title{Blow-up bounds under profile approximation\\
for superconformal wave equations}
\author{Firas Kaabi\\[3pt]
\small Faculty of Sciences of Tunis, LR21ES08\\
\small University of Tunis El Manar, Tunisia}
\date{}

\begin{document}
\maketitle

\begin{abstract}
For the focusing semilinear wave equation in the superconformal,
Sobolev-subcritical range, we prove an upper blow-up bound at the
ODE rate near a non-characteristic point, assuming approximation
by signed multiples of explicit stationary profiles in similarity variables.
The comparison parameters may vary with time, and the error is
measured only in a degenerate energy norm on a small extension
of the unit ball. The main step is a quantitative exclusion of
concentration at the boundary of the profile family. It follows
from the nonnegativity of the superconformal Hamza--Zaag energy
at admissible vertices and an upper perturbation estimate.
We also exclude families with an isolated dominant profile and,
under a noncancellation condition, clusters of comparable profiles.
For the latter, we derive a limiting energy in $\R^{N+1}$ which
accounts for the interactions between profiles.
\end{abstract}

\noindent\textbf{Keywords.}
Semilinear wave equation; blow-up rate; non-characteristic point;
self-similar profiles; Lyapunov functional.

\section{Introduction}

We consider the focusing semilinear wave equation
\begin{equation}\label{eq:nlw}
u_{tt}-\Delta u=|u|^{p-1}u,
\qquad N\geq3,\qquad
p_c:=1+\frac4{N-1}<p<p_S:=1+\frac4{N-2}.
\end{equation}
The initial data belong to
$H^1_{\mathrm{loc}}(\R^N)\times L^2_{\mathrm{loc}}(\R^N)$.
We work on the maximal influence domain
$\{(x,t):0\leq t<T(x)\}$, where $T$ is $1$-Lipschitz
\cite{Alinhac95}. Fix $x_0$ with $T_0=T(x_0)<\infty$.
The point $x_0$ is non-characteristic if, for some
$\delta_0\in(0,1)$ and $0\leq t_0<T_0$, the solution is defined on
\begin{equation}\label{eq:nc}
\{(x,t):t_0\leq t<T_0-\delta_0|x-x_0|\}.
\end{equation}
The associated lower bound on the blow-up graph is
\begin{equation}\label{eq:graphlower}
T(x_0+h)\geq T_0-\delta_0|h|
\qquad\text{for sufficiently small }h.
\end{equation}

The aim is an upper bound at the rate predicted by the ordinary
differential equation $U''=|U|^{p-1}U$. We obtain this bound under
an approximation hypothesis involving an explicit family of
self-similar solutions. The main step is to show that a profile
which approximates the solution sufficiently well cannot have
its parameter approach the boundary of that family.

Set
\begin{equation}\label{eq:parameters}
k=\frac2{p-1},\qquad a=\frac{N-1}{2},\qquad
\mu=a-k\in(0,\tfrac12),\qquad \alpha=-\mu,
\end{equation}
and introduce the similarity variables
\begin{equation}\label{eq:similarity}
y=\frac{x-x_0}{T_0-t},\qquad s=-\log(T_0-t),\qquad
w(y,s)=(T_0-t)^k u(x,t).
\end{equation}
We write $B_R=\{y\in\R^N:|y|<R\}$ and $B=B_1$.
By \eqref{eq:nc}, $w$ is defined on each fixed ball
$B_A$, $1<A<\delta_0^{-1}$, for all sufficiently large $s$.
The full similarity norm is
\begin{equation}\label{eq:M}
M(s)^2=\|w(s)\|_{H^1(B)}^2+\|w_s(s)\|_{L^2(B)}^2.
\end{equation}
Whether $M(s)$ stays bounded for every non-characteristic blow-up
point in \eqref{eq:nlw} is the general rate problem considered
here; it is listed as open in \cite[Section~4.1]{Z26}.

\subsection{Main results}

For $|d|<1$, define
\begin{equation}\label{eq:Wd}
W_d(y)=\kappa(d,y)
=\kappa_0\frac{(1-|d|^2)^{1/(p-1)}}{(1+d\cdot y)^{2/(p-1)}},
\qquad
\kappa_0=\left[\frac{2(p+1)}{(p-1)^2}\right]^{1/(p-1)}.
\end{equation}
We use the formula wherever $1+d\cdot y>0$.
This is the explicit stationary family in similarity variables,
denoted by $\kappa(d,y)$ in \cite[equation~(1.7)]{MZ16}.
It comes from applying Lorentz transformations to the spatially
homogeneous blow-up solution; the corresponding exact solutions are
\[
u_d(x,t)=\kappa_0(1-|d|^2)^{k/2}
\bigl(T_0-t+d\cdot(x-x_0)\bigr)^{-k}.
\]
The formula remains valid in the superconformal range.
The pairs $\sigma\lambda(W_d,0)$, with $\sigma\in\{-1,1\}$
and $\lambda>0$, will serve as comparison data.
For $\lambda\ne1$ they are not stationary solutions.

Write $\eta=1-|d|$. The error is measured in the degenerate energy norm
\begin{equation}\label{eq:X}
\|(f,g)\|_{\X_R}^2
=\int_{B_R}\left[g^2+f^2+R^2|\nabla f|^2
-(y\cdot\nabla f)^2\right]dy,
\qquad R\geq1,
\end{equation}
defined on $H^1(B_R)\times L^2(B_R)$.
Whenever $R|d|<1$, put
\begin{equation}\label{eq:error}
\E^R_{\lambda,d,\sigma}(f,g)
=\frac{\|(f,g)-\sigma\lambda(W_d,0)\|_{\X_R}}
{\lambda(1-|d|)^{\mu/2}}.
\end{equation}
The normalization is the size of the comparison profile in this norm
as $|d|\uparrow1$; see Lemma~\ref{lem:profile-size}.

For $m>0$, define
\begin{equation}\label{eq:qm}
q_\infty=\frac{2(N-2k)}{N+1}\in(0,1),
\qquad q_m=\max\{0,\,1-m^{p-1}(1-q_\infty)\}<1.
\end{equation}
Choose $c\in(q_m,1)$ and
\begin{equation}\label{eq:eta-cap}
0<\eta_c<\min\left\{1,\frac{\delta_0^{-1}-1}{c}\right\},
\qquad R(\eta)=1+c\min\{\eta,\eta_c\}.
\end{equation}
These choices ensure that the comparison profiles and the solution
are both defined on $B_{R(\eta)}$ at late times.
Define the best normalized error by
\begin{equation}\label{eq:minerror}
\E_m(s)=\inf_{\substack{|d|<1,\ \lambda\geq m\\ \sigma=\pm1}}
\E^{R(1-|d|)}_{\lambda,d,\sigma}(w(s),w_s(s)).
\end{equation}

\begin{theorem}[Blow-up bound under profile approximation]
\label{thm:bound}
Assume \eqref{eq:nc}, and fix $m,c,\eta_c$ as above.
There exist $\varepsilon,C>0$, depending only on
$N,p,m,c,\delta_0,\eta_c$, such that the hypothesis
\begin{equation}\label{eq:approximation}
\limsup_{s\to\infty}\E_m(s)<\varepsilon
\end{equation}
implies the following bound for some $t_1<T_0$:
\begin{equation}\label{eq:physical}
\begin{split}
&\tau^{k-N/2}\|u(t)\|_{L^2(B(x_0,\tau))}\\
&\quad+\tau^{k+1-N/2}
\left(\|\nabla u(t)\|_{L^2(B(x_0,\tau))}
+\|u_t(t)\|_{L^2(B(x_0,\tau))}\right)\leq C,
\end{split}
\end{equation}
where $\tau=T_0-t$ and $t_1\leq t<T_0$.
More precisely, for every sufficiently large $s$,
\begin{equation}\label{eq:pointwise-bound}
\E_m(s)<\varepsilon\quad\Longrightarrow\quad M(s)\leq C.
\end{equation}
Every comparison with
$\E^{R(\eta)}_{\lambda,d,\sigma}(w(s),w_s(s))<\varepsilon$
also satisfies $\eta\geq\eta_*$ and $\lambda\leq C$,
for some $\eta_*>0$ depending only on the same parameters.
\end{theorem}

The approximation in \eqref{eq:approximation} is an assumption on
the solution. No fixed comparison profile is prescribed, and the
parameters may vary with $s$ without continuity or convergence.
Only the lower bound $\lambda\geq m$ is imposed; the upper bound
on the amplitude follows from the equation. The error is controlled
solely by \eqref{eq:X}, with no additional smallness assumption in
$L^{p+1}$. The time $t_1$ may depend on the solution.

The exclusion used to prove Theorem~\ref{thm:bound} is independent
of the approximation hypothesis.

\begin{theorem}[Separation from concentrating comparison profiles]
\label{thm:exclusion}
Fix $m>0$, $c\in(q_m,1)$ and a non-characteristic margin
$\delta_0$. There exist $\eta_*,\varepsilon_*>0$, depending only on
$N,p,m,c,\delta_0$, such that every solution satisfying \eqref{eq:nc}
has a time $s_*$ for which
\begin{equation}\label{eq:main-exclusion}
\E^{1+c\eta}_{\lambda,d,\sigma}(w(s),w_s(s))
\geq\varepsilon_*
\end{equation}
for all $s\geq s_*$, $0<\eta=1-|d|<\eta_*$,
$\lambda\geq m$ and $\sigma=\pm1$.
The constants may be chosen so that
$1+c\eta_*<\delta_0^{-1}$.
\end{theorem}

The proofs of Theorems~\ref{thm:exclusion} and \ref{thm:bound}
are given in Sections~\ref{sec:single-proof} and \ref{sec:bound},
respectively. The observation ball in \eqref{eq:main-exclusion}
extends beyond $B$ by a distance proportional to $\eta$.
This region is available from the non-characteristic hypothesis
and is used in the change of vertex.

\subsection{Comparison with previous bounds}

For $1<p\leq p_c$, Merle and Zaag proved the upper bound at the
ODE rate without a profile approximation hypothesis
\cite{MZ03,MZ05c,MZ05g}. In the superconformal range, the general
bounds of Killip, Stovall and Vi\c{s}an \cite{KSV14} were improved
by Hamza and Zaag \cite{HZ13,HZ25}. To compare the latter result
with \eqref{eq:physical}, let
\begin{equation}\label{eq:local-time-energy}
\mathcal I_u(t)=
\int_t^{(t+T_0)/2}\int_{|x-x_0|<T_0-r}
\bigl(|\nabla u(x,r)|^2+|u_t(x,r)|^2\bigr)\,dx\,dr.
\end{equation}
Hamza and Zaag \cite[Theorem~1]{HZ25} prove, for every fixed
$\ell>0$ and at every non-characteristic point,
\begin{equation}\label{eq:HZ}
|\log(T_0-t)|^\ell\mathcal I_u(t)\leq C_\ell.
\end{equation}
Their result also controls a signed energy and a scaled $L^2$ norm.

Under \eqref{eq:approximation}, Theorem~\ref{thm:bound} gives
the full upper bound \eqref{eq:physical} at each late time. In
particular, it yields the stronger decay estimate
\begin{equation}\label{eq:power-improvement}
\mathcal I_u(t)\leq C(T_0-t)^{2\mu},
\qquad 2\mu=N-1-\frac4{p-1}>0.
\end{equation}
Indeed, squaring the derivative terms in \eqref{eq:physical}
and integrating gives
\[
\mathcal I_u(t)
\leq C\int_{\tau/2}^{\tau}r^{N-2k-2}\,dr
=\frac{C(1-2^{-2\mu})}{2\mu}\tau^{2\mu},
\qquad \tau=T_0-t.
\]
Thus a positive power of $T_0-t$ replaces the inverse logarithmic
bound on this class of solutions. The gain is in the rate and in
the pointwise-in-time control of the full local energy norm.
The additional hypothesis is the approximation on
$B_{R(1-|d|)}$ at the scale specified in \eqref{eq:error}.
The generality of \cite{HZ25} is greater: its hypotheses do not
include \eqref{eq:approximation}.

Theorem~\ref{thm:exclusion} provides the geometric information
behind this implication. It rules out accurate approximation by
profiles with $|d|\uparrow1$, where the full gradient norm diverges
although the degenerate energy norm tends to zero. Smallness
relative to the full gradient norm does not by itself give the
normalized error bound in \eqref{eq:error}.
Proposition~\ref{prop:log-gap} gives comparison data which satisfy
the numerical growth estimates \eqref{eq:HZ} and the accompanying
bounds of \cite{HZ25}, but are excluded by
Theorem~\ref{thm:exclusion}. This illustrates the information
supplied by the exclusion; the comparison data are not PDE
solutions.

\subsection{Historical background}

The existence of blow-up and the description of its rate are
distinct questions. Levine's concavity method establishes
finite-time blow-up from negative energy for broad classes of
nonlinear wave equations \cite{Levine74}. Caffarelli and Friedman
\cite{CF86} studied the geometry of the blow-up boundary under
positivity and structural assumptions. The constructions of
Kichenassamy and Littman \cite{KL93a,KL93b} near prescribed analytic
space-like surfaces, including for exponential nonlinearities,
further established the role of that geometry. Alinhac's monograph
\cite{Alinhac95} develops the maximal influence domain and local
blow-up mechanisms.

For the power equation, Antonini and Merle \cite{AM01} obtained
optimal estimates for positive solutions and introduced a
similarity-variable approach to the rate problem. The work of
Merle and Zaag cited above removed the positivity assumption and
treated the conformal exponent. The distinction between the
ranges is reflected in the weight
$(1-|y|^2)^{2/(p-1)-(N-1)/2}$: its exponent changes sign at $p_c$.
For $p>p_c$, Killip, Stovall and Vi\c{s}an \cite{KSV14}, in work
first circulated in 2012, used functionals associated with the
dilation identity for the nonlinear Klein--Gordon equation.
Hamza and Zaag \cite{HZ13} constructed a different Lyapunov
functional in similarity variables. Its nonlinear bulk dissipation
is central to their rate estimates and its sign is the starting
point of the present argument.

The description of profiles is more complete in one space
dimension. Merle and Zaag identified the family \eqref{eq:Wd}
at non-characteristic points \cite{MZ07}, proved openness of that
set and regularity of the blow-up curve \cite{MZ08}, and extended
the description to radial solutions away from the origin
\cite{MZ11r}. At characteristic points, several solitons occur;
their structure and the local finiteness of the characteristic
set were established in \cite{MZ12a,MZ12b}. C\^ote and Zaag
\cite{CZ13} refined these asymptotics and constructed solutions
with a prescribed number of solitons of alternating signs.

In higher dimensions, Merle and Zaag studied stability and
regularity at non-characteristic points with an explicit profile
\cite{MZ15}, and the dynamics near the stationary family
\cite{MZ16}, in the subconformal range. Their construction of a
nearly pyramidal blow-up surface \cite{MZpyramid} involves four
interacting profiles at an isolated characteristic point in two
dimensions. These results show why interactions between profiles
must be considered, but they do not provide a superconformal
profile decomposition.

A related question is stability near a prescribed self-similar
solution. Donninger and Sch\"orkhuber \cite{DS12} proved stability
of the spatially homogeneous blow-up in the energy topology for
radial data in dimension three and $1<p\leq3$. Ostermann
\cite{O24} obtained stability without symmetry restrictions in
all dimensions and for all superlinear powers, in suitable
higher-regularity spaces. Liu \cite{Liu26} treated the full
explicit family in the superconformal setting through spectral
equivalence under Lorentz transformations. These stability
results concern perturbations near a fixed profile. Our exclusion
is uniform as $|d|\uparrow1$, under the error and geometric
conditions in Theorem~\ref{thm:exclusion}.

\subsection{The argument and further results}

The proof uses the superconformal energy of Hamza and Zaag at
different admissible vertices. A change of vertex makes the energy
of a concentrating comparison profile negative. An upper
perturbation estimate preserves this sign under the error bound
\eqref{eq:error}, contradicting nonnegativity for the actual
solution. Section~\ref{sec:prelim} identifies the energy and gives
the change of variables; Sections~\ref{sec:threshold} and
\ref{sec:perturbation} establish the sign and perturbation
estimates. Section~\ref{sec:geometry} proves the two main theorems.

The remaining exclusions address sums of profiles.
Theorem~\ref{thm:multi} treats a family with an isolated dominant
member, allowing arbitrary signs and ratios of concentration
scales under an explicit relative summability bound.
Theorem~\ref{thm:cluster} treats a fixed number of comparable
profiles whose directions meet at the concentration scale,
under a noncancellation condition on the first singular group.
Its proof uses a limiting energy in $\R^{N+1}$.
Section~\ref{sec:comparison} examines the limits of growth
estimates and sign constraints taken alone.
Section~\ref{sec:horizons} then formulates the extraction and
evolution questions needed for a general upper bound.

\section{The superconformal energy}\label{sec:prelim}

With the convention that $(y\cdot\nabla)^2$ is the composition of two
operators, the similarity equation is
\begin{equation}\label{eq:sse}
w_{ss}+2y\cdot\nabla w_s+(2k+1)w_s
=\LL w-k(k+1)w+|w|^{p-1}w,
\end{equation}
where
\begin{equation}\label{eq:operator}
\begin{split}
\LL f
&=\Delta f-(y\cdot\nabla)^2f-(2k+1)y\cdot\nabla f\\
&=(1-|y|^2)^{-\alpha}
\diver\left((1-|y|^2)^\alpha(I-y\otimes y)\nabla f\right).
\end{split}
\end{equation}
The identity $N+2\alpha=2k+1$ explains the coefficient in this expression.

For a pair $(f,g)\in H^1(B)\times L^2(B)$, define the corrected energy
\begin{equation}\label{eq:H}
\begin{split}
H[f,g]=\int_B\bigg[
&\frac12g^2+\frac12\bigl(|\nabla f|^2-(y\cdot\nabla f)^2\bigr)
+\frac{k(k+1)}2f^2\\
&-\frac{|f|^{p+1}}{p+1}+\alpha fg-\frac{\alpha N}{2}f^2
\bigg]dy.
\end{split}
\end{equation}
With $\mu$ equal to the parameter $\eta$ of
\cite[equations~(1.18)--(1.19)]{HZ13}, this is exactly the
superconformal energy $E=E_0+I$ of that paper. In particular, the
correction $-\mu fg+\mu Nf^2/2$ is retained. Their decreasing
functional is
\begin{equation}\label{eq:HZ-functional}
\mathcal F_{\mathrm{HZ}}(s)
=e^{-2\mu s}H[w(s),w_s(s)],
\end{equation}
as defined in \cite[equation~(1.20)]{HZ13}.
We use the sign of $H$; monotonicity applies to
$\mathcal F_{\mathrm{HZ}}$, not to $H$ itself.

The energy $H$ is finite on $H^1(B)\times L^2(B)$ by \eqref{eq:nlw}, and is
continuous for the sum of the $\X_1$ and $L^{p+1}$ norms.
Write
\begin{equation}\label{eq:quadratic}
\mathcal Q[f,g]=H[f,g]+\frac1{p+1}\int_B|f|^{p+1}\,dy.
\end{equation}
Completing the square gives
\begin{equation}\label{eq:Q-square}
2\mathcal Q[f,g]
=\int_B\left[(g-\mu f)^2+
|\nabla f|^2-(y\cdot\nabla f)^2+\Lambda f^2\right]dy,
\qquad
\Lambda=k(k+1)+\mu N-\mu^2>0.
\end{equation}
Thus there are constants $c_Q,C_Q>0$, depending only on $N,p$, with
\begin{equation}\label{eq:Q-equivalence}
c_Q\|(f,g)\|_{\X_1}^2
\leq\mathcal Q[f,g]\leq C_Q\|(f,g)\|_{\X_1}^2.
\end{equation}
In particular the associated bilinear form is bounded on $\X_1$.
We will use an upper perturbation bound for $H$, rather than
continuity of $H$ in $\X_1$.

A vertex $(\widehat x,\widehat T)$ is called admissible if
$0<\widehat T\leq T(\widehat x)$. Its backward light cone lies in the
maximal influence domain. We recall the two facts from
\cite[Proposition~1.2 and equation~(2.4)]{HZ13} used below:
for a classical solution on a neighbourhood of $\overline B$,
\begin{equation}\label{eq:lyapunov}
\frac{d}{ds}\bigl(e^{-2\mu s}H[w,w_s]\bigr)
=-e^{-2\mu s}\left[
\int_{\partial B}(w_s-\mu w)^2\,dS
+\mu\frac{p-1}{p+1}\int_B|w|^{p+1}\,dy
\right],
\end{equation}
and, for the energy solution expressed at any admissible vertex,
\begin{equation}\label{eq:nonnegative}
H[\widehat w(\widehat s),\widehat w_{\widehat s}(\widehat s)]
\geq0.
\end{equation}
The trace identity will be applied only to explicit smooth profiles.
For the actual solution we use the energy-class sign statement
\eqref{eq:nonnegative}, so no classical boundary trace is assumed.

For $\nu<1-|d|$, let
\begin{equation}\label{eq:shifted}
W_{d,\nu}(y)=c_d(1+d\cdot y-\nu)^{-k},\qquad
c_d=\kappa_0(1-|d|^2)^{k/2},
\end{equation}
and define
\[
\mathbf W_{d,\nu}=(W_{d,\nu},Z_{d,\nu}),\qquad
Z_{d,\nu}=\nu\partial_\nu W_{d,\nu},\qquad
h_d(\nu)=H[\mathbf W_{d,\nu}].
\]
The curve $\nu(s)=\nu_0e^s$ gives an exact solution of
\eqref{eq:sse} on every interval on which the denominator is positive.
Thus $\mathbf W_{d,\nu}$ denotes the phase data of a time-shifted
explicit solution. Only $\nu=0$ gives the stationary pair $(W_d,0)$.

Here is the change of coordinates at a single physical time. Write
$\tau=T_0-t$, $\widehat\tau=\widehat T-t$, and assume both are positive.
Set
\begin{equation}\label{eq:map}
b=\frac{\widehat\tau}{\tau},\qquad
\gamma=\frac{\widehat x-x_0}{\tau},\qquad
y=\gamma+bY.
\end{equation}
For data $(f,g)=(w,w_s)$ at the original similarity time, define
\begin{equation}\label{eq:transport}
\begin{split}
\mathcal T_{b,\gamma}^{(1)}(f,g)(Y)
&=b^k f(y),\\
\mathcal T_{b,\gamma}^{(2)}(f,g)(Y)
&=b^k\left[
b\,g(y)+b(y-Y)\cdot\nabla f(y)+k(b-1)f(y)
\right].
\end{split}
\end{equation}
\begin{lemma}[Change of vertex]\label{lem:vertex-identity}
At the same physical time, the similarity data at the two vertices
satisfy
\[
(\widehat w,\widehat w_{\widehat s})
=\mathcal T_{b,\gamma}(w,w_s).
\]
If
\begin{equation}\label{eq:nu}
\nu=\frac{\widehat T-T_0-d\cdot(\widehat x-x_0)}
{\widehat T-t},
\end{equation}
then, wherever the profile denominator is positive,
\begin{equation}\label{eq:exact-transport}
\mathcal T_{b,\gamma}(W_d,0)=\mathbf W_{d,\nu}.
\end{equation}
\end{lemma}

\begin{proof}
The first component follows from the definition of the similarity
variables. For the second, use
\[
\tau^{k+1}u_t=w_s+kw+y\cdot\nabla w,\qquad
\widehat w_{\widehat s}
=\widehat\tau^{k+1}u_t-k\widehat w-Y\cdot\nabla_Y\widehat w.
\]
Substitution gives \eqref{eq:transport}. For energy data these are
identities in $L^2$ on each controlled ball. Finally,
$1+d\cdot(\gamma+bY)=b(1+d\cdot Y-\nu)$;
inserting $W_d$ into both components gives
\eqref{eq:exact-transport}.
\end{proof}

\section{Energy of the shifted explicit profiles}\label{sec:threshold}

The profiles $W_d$ are defined in \eqref{eq:Wd}, and their shifted
phase data $\mathbf W_{d,\nu}$ in \eqref{eq:shifted}.
This section determines the sign of $h_d(\nu)=H[\mathbf W_{d,\nu}]$.
The threshold and its asymptotics will be used to choose the
vertex in the exclusion argument.

For $j>0$, write
\[
M_j(d,\nu)=\int_B(1+d\cdot y-\nu)^{-j}\,dy.
\]
The constant denoted by $M_j$ is a moment, and is unrelated to $M(s)$.

\begin{lemma}[Moment formula]\label{lem:closed}
For all $|d|<1$ and $\nu<1-|d|$,
\begin{equation}\label{eq:closed}
\begin{split}
h_d(\nu)=c_d^2\bigg[
&k^2(|d|^2-1+\nu)M_{2k+2}\\
&+\bigl(k^2(1-\nu)+\alpha k\nu\bigr)M_{2k+1}
+\frac{k-\alpha N}{2}M_{2k}
\bigg].
\end{split}
\end{equation}
Moreover,
\begin{equation}\label{eq:d-zero}
h_0(\nu)=\frac{\kappa_0^2|B|(N-1)}4
(1-\nu)^{-2k-1}(N-2k-N\nu).
\end{equation}
\end{lemma}

\begin{proof}
Put $U=1+d\cdot y-\nu$. Then
\[
\nabla W_{d,\nu}=-kc_dU^{-k-1}d,\qquad
Z_{d,\nu}=kc_d\nu U^{-k-1},
\]
and
\[
|\nabla W_{d,\nu}|^2-(y\cdot\nabla W_{d,\nu})^2
=k^2c_d^2U^{-2k-2}\bigl(|d|^2-(d\cdot y)^2\bigr).
\]
Insert these expressions into \eqref{eq:H}, substitute
$d\cdot y=U-(1-\nu)$, and use
\[
c_d^{p+1}=c_d^2k(k+1)(1-|d|^2),\qquad
\frac{k(k+1)}{p+1}=\frac{k^2}{2}.
\]
Collecting the three powers of $U$ gives \eqref{eq:closed}.
For $d=0$, $M_j=|B|(1-\nu)^{-j}$; simplification with
$\alpha=k-(N-1)/2$ gives \eqref{eq:d-zero}.
\end{proof}

\begin{proposition}[Unique sign threshold]\label{prop:threshold}
For every $|d|<1$, $h_d$ has exactly one zero
$\nu_*(d)$ in its full domain $(-\infty,1-|d|)$.
It belongs to $(0,1-|d|)$, and
\[
h_d(\nu)>0\ \text{for }\nu<\nu_*(d),\qquad
h_d(\nu)<0\ \text{for }\nu_*(d)<\nu<1-|d|.
\]
The map $d\mapsto\nu_*(d)$ is smooth on the open unit ball and invariant
under rotations. At the origin,
\[
\nu_*(0)=1-\frac{2k}{N}.
\]
\end{proposition}

\begin{proof}
Define the strictly positive quantity
\begin{equation}\label{eq:dissipation}
\D_d(\nu)=
\int_{\partial B}(Z_{d,\nu}-\mu W_{d,\nu})^2\,dS
+\mu\frac{p-1}{p+1}\int_BW_{d,\nu}^{p+1}\,dy.
\end{equation}
Applying \eqref{eq:lyapunov} to the exact trajectory
$\nu_s=\nu$ gives
\begin{equation}\label{eq:scalar-ode}
\nu h_d'(\nu)-2\mu h_d(\nu)=-\D_d(\nu).
\end{equation}
At the stationary point, the same identity yields
\begin{equation}\label{eq:stationary-positive}
h_d(0)=\frac{\D_d(0)}{2\mu}>0.
\end{equation}
For $\nu>0$,
\begin{equation}\label{eq:monotone}
\frac{d}{d\nu}\bigl(\nu^{-2\mu}h_d(\nu)\bigr)
=-\nu^{-2\mu-1}\D_d(\nu)<0.
\end{equation}

We next check the negative sign at the upper endpoint. If
$\rho=|d|\in(0,1)$, rotation and integration in the transverse
variables give
\begin{equation}\label{eq:moment-integral}
M_j=K_a\int_{-1}^1(1-t^2)^a(1+\rho t-\nu)^{-j}\,dt,
\qquad K_a=\frac{\pi^a}{\Gamma(a+1)}.
\end{equation}
As $\epsilon=1-\rho-\nu\downarrow0$, the integral near $t=-1$
has order $\epsilon^{a+1-j}$ for $j>a+1$, order
$\log(1/\epsilon)$ for $j=a+1$, and stays bounded for $j<a+1$.
Since $2k>a$ in \eqref{eq:nlw}, both $M_{2k+1}$ and
$M_{2k+2}$ diverge, and
\[
\frac{M_{2k+1}}{M_{2k+2}}\to0,\qquad
\frac{M_{2k}}{M_{2k+2}}\to0.
\]
The coefficient of $M_{2k+2}$ in \eqref{eq:closed} tends to
$k^2\rho(\rho-1)<0$. Hence $h_d(\nu)\to-\infty$.
For $d=0$ the same conclusion follows from \eqref{eq:d-zero}.
Together with \eqref{eq:stationary-positive} and
\eqref{eq:monotone}, this proves existence, uniqueness and the sign
description on the positive branch.

For the negative branch, put $g(x)=h_d(-x)$, $x>0$.
Equation \eqref{eq:scalar-ode} gives
\[
\frac{d}{dx}\bigl(x^{-2\mu}g(x)\bigr)
=-x^{-2\mu-1}\D_d(-x)<0.
\]
At fixed $d$, the explicit profiles and their velocity satisfy
$W_{d,-x}=O(x^{-k})$ and $Z_{d,-x}=O(x^{-k})$ uniformly on
$\overline B$. Thus $g(x)=O(x^{-2k})$, and
$x^{-2\mu}g(x)\to0$ as $x\to\infty$. A strictly decreasing function
with this limit is strictly positive. Hence $h_d(\nu)>0$ for $\nu<0$.

Finally, at the zero,
\[
h_d'(\nu_*(d))=-\frac{\D_d(\nu_*(d))}{\nu_*(d)}<0.
\]
The function $(d,\nu)\mapsto h_d(\nu)$ is smooth in its open domain,
because its denominator is bounded away from zero locally.
The implicit function theorem and rotation invariance give the
remaining assertions.
\end{proof}

\subsection{Asymptotics of the threshold}

Set $\beta=2k-a>0$ and recall $q_\infty$ from \eqref{eq:qm}.
For $0<\eta<1/2$, define
\begin{equation}\label{eq:remainder}
r_\beta(\eta)=
\begin{cases}
\eta^\beta,&0<\beta<1,\\
\eta\log(1/\eta),&\beta=1,\\
\eta,&\beta>1.
\end{cases}
\end{equation}

\begin{lemma}[Uniform moment estimates]\label{lem:moments}
Let $\eta=1-|d|$ and $\nu=q\eta$, with $q$ in a fixed compact
$K\subset[0,1)$. Uniformly on $K$,
\begin{align}
M_{2k+1}
&=K_a\eta^{a-2k}
\left[
2^a(1-q)^{a-2k}\mathrm B(a+1,2k-a)
+O_K(r_\beta(\eta))
\right], \label{eq:moment-one}\\
M_{2k+2}
&=K_a\eta^{a-2k-1}
\left[
2^a(1-q)^{a-2k-1}\mathrm B(a+1,2k+1-a)
+O_K(\eta)
\right], \label{eq:moment-two}\\
\eta^{-\mu}c_d^2M_{2k}
&=O_K(r_\beta(\eta)). \label{eq:moment-zero}
\end{align}
Here $\mathrm B$ denotes the beta function.
\end{lemma}

\begin{proof}
In \eqref{eq:moment-integral}, use $t=-1+\eta x$.
Then
\begin{equation}\label{eq:scaled-moment}
M_j=K_a\eta^{a+1-j}I_j(\eta,q),\qquad
I_j(\eta,q)=\int_0^{2/\eta}
\frac{x^a(2-\eta x)^a}
{(1-q+(1-\eta)x)^j}\,dx.
\end{equation}
For $j>a+1$, the limiting integral is
\[
I_j(0,q)
=2^a(1-q)^{a+1-j}\mathrm B(a+1,j-a-1).
\]
The denominators are uniformly comparable to $1+x$ for $q\in K$.
Replacing $1-\eta$ by $1$ in the denominator gives a relative
$O_K(\eta)$ error. Since $a\geq1$,
\[
|(2-\eta x)^a-2^a|\leq C_a\min\{\eta x,1\}
\quad(0\leq x\leq2/\eta).
\]
For $j=2k+1$, the tail of the limiting integrand is
$O_K(x^{-1-\beta})$. The remaining error is bounded by
\[
C_K\left[
\eta+\eta\int_1^{2/\eta}x^{-\beta}\,dx
+\int_{2/\eta}^{\infty}x^{-1-\beta}\,dx
\right]
\leq C_Kr_\beta(\eta).
\]
For $j=2k+2$, the tail is $O_K(x^{-2-\beta})$; the same calculation
gives an $O_K(\eta)$ error. This proves the first two assertions.

For the last one, the integral defining $M_{2k}$ is bounded if
$2k<a+1$, is $O_K(\log(1/\eta))$ if $2k=a+1$, and is
$O_K(\eta^{a+1-2k})$ if $2k>a+1$.
Using $c_d^2\asymp\eta^k$ and $k-\mu=\beta$ yields
\eqref{eq:moment-zero}.
\end{proof}

\begin{proposition}[Asymptotic sign and location of the zero]
\label{prop:asymptotic}
Uniformly for $q$ in compact subsets of $[0,1)$,
\begin{equation}\label{eq:asymptotic}
\eta^{-\mu}h_d(q\eta)=G(q)+O_K(r_\beta(\eta)),
\end{equation}
where
\begin{align}
G(q)&=C_{N,k}(1-q)^{a-2k}F(q),\label{eq:G}\\
F(q)&=1-\frac{(2-q)(2k-a)}{(2k+1)(1-q)},\label{eq:F}\\
C_{N,k}&=\kappa_0^2\,2^{k+a}k^2K_a\,
\mathrm B(a+1,2k-a)>0.\label{eq:C}
\end{align}
The unique zero of $F$ in $(0,1)$ is $q_\infty$, and $F<0$ on
$(q_\infty,1)$. Moreover,
\begin{equation}\label{eq:root-asymptotic}
\nu_*(d)=q_\infty\eta+O\bigl(\eta r_\beta(\eta)\bigr),
\qquad \eta=1-|d|\downarrow0.
\end{equation}
\end{proposition}

\begin{proof}
The coefficients in \eqref{eq:closed} satisfy
\[
|d|^2-1+q\eta=\eta(q-2)+\eta^2,\qquad
k^2(1-q\eta)+\alpha kq\eta=k^2+O_K(\eta),
\]
and $c_d^2=\kappa_0^22^k\eta^k(1+O(\eta))$.
Insert Lemma~\ref{lem:moments} in \eqref{eq:closed}. The first two
terms both have order $\eta^\mu$, while the third is
$O_K(\eta^\mu r_\beta(\eta))$.
The beta identity
\[
\frac{\mathrm B(a+1,2k+1-a)}{\mathrm B(a+1,2k-a)}
=\frac{2k-a}{2k+1}
\]
gives \eqref{eq:asymptotic}--\eqref{eq:C}.

The function $F$ is strictly decreasing, since
$(2-q)/(1-q)=1+(1-q)^{-1}$. Also,
\[
F(0)=\frac{N-2k}{2k+1}>0,\qquad
F(q)\longrightarrow-\infty\quad(q\uparrow1).
\]
Solving $F(q)=0$ gives
$q_\infty=2(N-2k)/(N+1)$.
It lies in $(0,1)$ because $2k<N$ and $2k>a$.

For any small fixed $\epsilon>0$, \eqref{eq:asymptotic} gives opposite
signs at $(q_\infty-\epsilon)\eta$ and
$(q_\infty+\epsilon)\eta$. The unique zero from
Proposition~\ref{prop:threshold} lies between them. Thus
$\nu_*(d)/\eta\to q_\infty$.
Choose a fixed compact neighbourhood of $q_\infty$ contained in $(0,1)$.
For small $\eta$, the zero lies in this neighbourhood, and
\eqref{eq:asymptotic} gives
\[
\left|G\left(\frac{\nu_*(d)}{\eta}\right)\right|
\leq C r_\beta(\eta).
\]
Since $G'(q_\infty)<0$, the mean value theorem proves
\eqref{eq:root-asymptotic}.
\end{proof}

\subsection{Dependence on the amplitude}

Fix an amplitude lower bound $m>0$, with $q_m$ as in \eqref{eq:qm}.
Define
\begin{equation}\label{eq:PQ}
P_d(\nu)=\frac1{p+1}\int_BW_{d,\nu}^{p+1}\,dy,\qquad
Q_d(\nu)=\mathcal Q[\mathbf W_{d,\nu}]=h_d(\nu)+P_d(\nu).
\end{equation}
Homogeneity gives the exact identity
\begin{equation}\label{eq:amplitude-identity}
H[\lambda\mathbf W_{d,\nu}]
=\lambda^2Q_d(\nu)-\lambda^{p+1}P_d(\nu)
\qquad(\lambda>0).
\end{equation}

\begin{proposition}[Amplitude-dependent sign margin]\label{prop:amplitude}
Let $m>0$, and let $K$ be a nonempty compact subset of $(q_m,1)$,
where $q_m$ is given by \eqref{eq:qm}.
There are $b_0,\eta_0>0$, depending only on $N,p,m,K$, such that
\begin{equation}\label{eq:amplitude-negative}
H[\lambda\mathbf W_{d,q\eta}]
\leq-2b_0\lambda^{p+1}\eta^\mu
\end{equation}
whenever $0<\eta<\eta_0$, $q\in K$, and $\lambda\geq m$.
\end{proposition}

\begin{proof}
Put $A(q)=C_{N,k}(1-q)^{-\beta}/(2k+1)$.
From \eqref{eq:moment-two} and
$k(k+1)/(p+1)=k^2/2$, we obtain
\begin{align}
\eta^{-\mu}P_d(q\eta)
&=A(q)\frac{\beta}{1-q}+O_K(r_\beta(\eta)),\label{eq:P-limit}\\
\eta^{-\mu}Q_d(q\eta)
&=A(q)(a+1)+O_K(r_\beta(\eta)).\label{eq:Q-limit}
\end{align}
Indeed,
$P_d=(k^2/2)c_d^2(1-|d|^2)M_{2k+2}$,
so the first line follows from the beta identity used above.
Adding it to \eqref{eq:asymptotic} gives the second line, because
$2k+1-\beta=a+1$.

Dividing \eqref{eq:amplitude-identity} by
$\lambda^{p+1}\eta^\mu$ gives, uniformly for $\lambda\geq m$,
\begin{equation}\label{eq:amplitude-asymptotic}
\frac{H[\lambda\mathbf W_{d,q\eta}]}{\lambda^{p+1}\eta^\mu}
=A(q)\left[(a+1)\lambda^{1-p}-\frac{\beta}{1-q}\right]
+O_{m,K}(r_\beta(\eta)).
\end{equation}
Here the remainder is uniform even for unbounded $\lambda$, since
$\lambda^{1-p}\leq m^{1-p}$.
The identity
\begin{equation}\label{eq:beta-ratio}
\frac{\beta}{a+1}=1-q_\infty
\end{equation}
shows that the leading expression is strictly negative, uniformly
on $K$ and $\lambda\geq m$: the condition is
$1-q<m^{p-1}(1-q_\infty)$.
Choose a fixed negative margin and then reduce $\eta_0$ to absorb
the remainder. This proves \eqref{eq:amplitude-negative}.
\end{proof}

\section{Perturbations in the degenerate energy norm}\label{sec:perturbation}

\begin{lemma}[A Sobolev inequality by spherical lifting]\label{lem:sobolev}
Set
\[
r=\frac{2(N+1)}{N-1},\qquad r'=\frac{2(N+1)}{N+3}.
\]
For every $f\in H^1(B)$,
\begin{equation}\label{eq:sobolev}
\|f\|_{L^r(B)}\leq C_N\|(f,0)\|_{\X_1}.
\end{equation}
\end{lemma}

\begin{proof}
View $\mathbb S^{N+1}$ as the unit sphere in
$\R^N\times\R^2$. For $f$ smooth on $\overline B$, lift it to
$F(y,z)=f(y)$ on this sphere. Away from $z=0$, use the coordinates
\[
(y,\theta)\longmapsto
\left(y,\sqrt{1-|y|^2}\cos\theta,
          \sqrt{1-|y|^2}\sin\theta\right).
\]
The induced metric has $y$ block
$I+y\otimes y/(1-|y|^2)$ and $\theta$ block $1-|y|^2$.
Its determinant is one and its inverse $y$ block is $I-y\otimes y$.
Consequently,
\begin{align*}
\int_{\mathbb S^{N+1}}|F|^r\,dS
&=2\pi\int_B|f|^r\,dy,\\
\int_{\mathbb S^{N+1}}\bigl(|\nabla_{\mathbb S}F|^2+F^2\bigr)\,dS
&=2\pi\int_B\bigl(|\nabla f|^2-(y\cdot\nabla f)^2+f^2\bigr)\,dy.
\end{align*}
The omitted set $z=0$ has surface measure zero.
The Sobolev inequality in dimension $N+1$ on the sphere, obtained
from the Euclidean inequality in a finite atlas and a partition of
unity, has exponent $r$ and gives \eqref{eq:sobolev}.
Finally, smooth approximation in $H^1(B)$ converges in $\X_1$.
The inequality makes these approximants Cauchy in $L^r$; their
$L^2$ limit is $f$, which proves the assertion for $H^1(B)$.
\end{proof}

\begin{lemma}[An upper perturbation estimate]\label{lem:upper-perturbation}
Let $\mathbf S=(S_1,S_2)\in H^1(B)\times L^2(B)$ satisfy
$|S_1|^{p-1}S_1\in L^{r'}(B)$. For
$\mathbf h=(f,g)\in H^1(B)\times L^2(B)$ and
$D=\|\mathbf h\|_{\X_1}$,
\begin{equation}\label{eq:general-upper}
H[\mathbf S+\mathbf h]\leq H[\mathbf S]
+C\left[
\left(\|\mathbf S\|_{\X_1}
+\||S_1|^{p-1}S_1\|_{r'}\right)D+D^2\right],
\end{equation}
where $C$ depends only on $N,p$.
\end{lemma}

\begin{proof}
By \eqref{eq:Q-equivalence}, the quadratic part satisfies
\[
\mathcal Q[\mathbf S+\mathbf h]-\mathcal Q[\mathbf S]
\leq C\bigl(\|\mathbf S\|_{\X_1}D+D^2\bigr).
\]
Convexity gives, for all real $S_1,f$,
\[
|S_1+f|^{p+1}\geq |S_1|^{p+1}
+(p+1)|S_1|^{p-1}S_1f.
\]
H\"older's inequality and Lemma~\ref{lem:sobolev} bound the absolute
value of its linear contribution by
$C\||S_1|^{p-1}S_1\|_{r'}D$.
Combining these estimates proves \eqref{eq:general-upper}.
All potential integrals are finite because $p<p_S$.
\end{proof}

\begin{lemma}[Sizes of the profiles]\label{lem:profile-size}
For a compact $K\subset[0,1)$ and $\nu=q\eta$,
\begin{equation}\label{eq:profile-size}
\|\mathbf W_{d,q\eta}\|_{\X_1}
\leq C_K\eta^{\mu/2},\qquad
\|W_{d,q\eta}\|_{L^{p+1}(B)}
\leq C_K\eta^{\mu/(p+1)}
\end{equation}
when $\eta$ is small. In addition, for each fixed $c\in(0,1)$,
\begin{equation}\label{eq:collar-size}
\|(W_d,0)\|_{\X_{1+c\eta}}\leq C_c\eta^{\mu/2},
\qquad
\|W_d\|_{L^{p+1}(B_{1+c\eta})}
\leq C_c\eta^{\mu/(p+1)}.
\end{equation}
Finally,
\begin{equation}\label{eq:H1-large}
\|\nabla W_d\|_{L^2(B)}^2\asymp\eta^{\mu-1}.
\end{equation}
\end{lemma}

\begin{proof}
The squared velocity norm is
$c_d^2k^2q^2\eta^2M_{2k+2}=O_K(\eta^{\mu+1})$.
The mass is $O_K(\eta^\mu)$ by \eqref{eq:moment-zero}.
After rotation, the degenerate gradient integral is
\[
c_d^2k^2|d|^2K_a
\int_{-1}^1
\frac{(1-t^2)^{a+1}}{(1+|d|t-q\eta)^{2k+2}}\,dt.
\]
The substitution in \eqref{eq:scaled-moment} gives
$O_K(\eta^{k+a-2k})=O_K(\eta^\mu)$;
the resulting integral is uniformly bounded because $2k>a$.
Furthermore,
\[
\int_BW_{d,q\eta}^{p+1}\,dy
=c_d^2k(k+1)(1-|d|^2)M_{2k+2}
=O_K(\eta^\mu).
\]
This proves \eqref{eq:profile-size}.

For \eqref{eq:collar-size}, put $R=1+c\eta$ and $d'=Rd$.
Then
\[
1-|d'|=(1-c)\eta+c\eta^2\asymp_c\eta,\qquad
W_d(RY)=A_\eta W_{d'}(Y),
\]
where
\[
A_\eta=\left(\frac{1-|d|^2}{1-R^2|d|^2}\right)^{k/2}
\]
is bounded above and below by positive constants depending on $c$.
Under $y=RY$, the degenerate gradient expression satisfies
\[
R^2|\nabla_y f|^2-(y\cdot\nabla_y f)^2
=|\nabla_Y(f(RY))|^2-(Y\cdot\nabla_Y(f(RY)))^2.
\]
The Jacobian is $R^N$. Apply \eqref{eq:profile-size} to $d'$ and $q=0$.
Finally,
$\|\nabla W_d\|_2^2=k^2c_d^2|d|^2M_{2k+2}(d,0)$;
\eqref{eq:moment-two} gives \eqref{eq:H1-large}.
\end{proof}

\begin{lemma}[The nonlinear profile as a bounded linear functional]
\label{lem:dual}
For every compact $K\subset[0,1)$ and sufficiently small $\eta$,
\begin{equation}\label{eq:dual}
\|W_{d,q\eta}^{p}\|_{L^{r'}(B)}
\leq C_K\eta^{\mu/2}
\qquad(q\in K).
\end{equation}
Thus, for $f\in H^1(B)$,
\begin{equation}\label{eq:pairing}
\left|\int_B W_{d,q\eta}^p f\,dy\right|
\leq C_K\eta^{\mu/2}\|(f,0)\|_{\X_1}.
\end{equation}
\end{lemma}

\begin{proof}
For every fixed $j>a+1$, the scaled integral
\eqref{eq:scaled-moment} is bounded uniformly on $K$, so
$M_j(d,q\eta)\leq C_K\eta^{a+1-j}$.
Take $j=kp r'=(k+2)r'$. This exponent exceeds $a+1$:
indeed $k>(N-2)/2$ implies $k+2>(N+3)/4$.
As $c_d^p\asymp\eta^{(k+2)/2}$,
\[
\|W_{d,q\eta}^p\|_{r'}
=c_d^p M_{(k+2)r'}^{1/r'}
\leq C_K\eta^{(a+1)/r'-(k+2)/2}
=C_K\eta^{\mu/2}.
\]
The last equality uses $(a+1)/r'=(N+3)/4$.
H\"older's inequality and Lemma~\ref{lem:sobolev} give
\eqref{eq:pairing}.
\end{proof}

\begin{proposition}[One-sided perturbative sign criterion]
\label{prop:perturbative}
Fix $m>0$ and a nonempty compact $K\subset(q_m,1)$.
There exist $b_0,\epsilon_0,\eta_0>0$, depending only on $N,p,m,K$,
such that the following holds.
If $0<\eta=1-|d|<\eta_0$, $q\in K$, $\lambda\geq m$,
$\sigma=\pm1$, and $(v,z)\in H^1(B)\times L^2(B)$ satisfy
\begin{equation}\label{eq:small-X}
\|(v,z)-\sigma\lambda\mathbf W_{d,q\eta}\|_{\X_1}
\leq\epsilon_0\lambda\eta^{\mu/2},
\end{equation}
then
\begin{equation}\label{eq:negative-neighbourhood}
H[v,z]\leq-b_0\lambda^{p+1}\eta^\mu<0.
\end{equation}
No smallness of the $L^{p+1}$ error is required.
\end{proposition}

\begin{proof}
Choose the constants in Proposition~\ref{prop:amplitude}.
By the symmetry $H[-v,-z]=H[v,z]$, we can assume $\sigma=1$.
Write $(f,g)=(v,z)-\lambda\mathbf W_{d,q\eta}$ and
$D=\|(f,g)\|_{\X_1}$.
Lemmas~\ref{lem:profile-size} and \ref{lem:dual}, applied in
Lemma~\ref{lem:upper-perturbation}, yield
\begin{equation}\label{eq:one-sided}
H[v,z]\leq H[\lambda\mathbf W_{d,q\eta}]
+C_K\left[(\lambda+\lambda^p)\eta^{\mu/2}D+D^2\right].
\end{equation}

Under \eqref{eq:small-X}, the error on the right is at most
\[
C_K\left[(\epsilon_0+\epsilon_0^2)\lambda^2
+\epsilon_0\lambda^{p+1}\right]\eta^\mu
\leq C_{m,K}(\epsilon_0+\epsilon_0^2)
\lambda^{p+1}\eta^\mu.
\]
Choose $\epsilon_0$ so that this is at most
$b_0\lambda^{p+1}\eta^\mu$, and apply
\eqref{eq:amplitude-negative}.
\end{proof}

\begin{remark}[Failure of the nonlinear embedding]
\label{rem:strict-weakening}
The embedding $\X_1\hookrightarrow L^{p+1}$ is false in the
superconformal range. Choose nonzero smooth cutoffs
$\chi,\psi$ supported, respectively, in $(1,2)$ and
$\{z'\in\R^{N-1}:|z'|<1\}$, and write $y=(y_1,y')$.
For small $h>0$ set
\[
f_h(y)=h^{-(N-1)/2+\theta}
\chi\left(\frac{1+y_1}{h^2}\right)
\psi\left(\frac{y'}h\right),
\qquad
0<\theta<\frac{N-1}{2}-\frac{N+1}{p+1}.
\]
The interval for $\theta$ is nonempty because $p>p_c$.
These smooth functions are supported inside $B$.
Their support has volume of order $h^{N+1}$; on it
$1-y_1^2=O(h^2)$.
The positive semidefinite quadratic form $I-y\otimes y$ therefore
gives
\[
\|(f_h,0)\|_{\X_1}=O(h^\theta)\longrightarrow0,\qquad
\|f_h\|_{p+1}\asymp
h^{\theta+(N+1)/(p+1)-(N-1)/2}\longrightarrow\infty.
\]
The same construction scales to each fixed $B_R$.
This gives perturbations small in $\X_R$ whose nonlinear norm
is arbitrarily large.
\end{remark}

The sign criterion applies to actual solutions only when the new
vertex is admissible and the data are controlled on its observation
ball. The next section obtains both properties from
\eqref{eq:nc} and the collar in Theorem~\ref{thm:exclusion}.

\section{Proofs of the main results}\label{sec:geometry}

\subsection{Exclusion of concentration}\label{sec:single-proof}

Fix $m>0$ and $c\in(q_m,1)$. We use the comparison error
$\E^R_{\lambda,d,\sigma}$ defined in \eqref{eq:error}.
Choose
\begin{equation}\label{eq:choices}
q=\frac{q_m+c}{2},\qquad
\delta=\frac{1+\delta_0}{2}.
\end{equation}
For $\rho=|d|=1-\eta$ close to one, put $e=d/\rho$ and
\begin{equation}\label{eq:chosen-map}
\xi_\eta=\frac{q\eta}{\rho-\delta+\delta q\eta},
\qquad
b_\eta=1-\delta\xi_\eta,\qquad
\gamma_\eta=-\xi_\eta e,\qquad R_\eta=1+c\eta.
\end{equation}
The denominator is positive for small $\eta$.

\begin{lemma}[Geometry of the observation balls]\label{lem:geometry}
There are $\eta_g>0$, $\kappa>0$ and $C_g>0$, depending only on
$c,q,\delta$, such that, for $0<\eta<\eta_g$,
\begin{equation}\label{eq:geom-bounds}
\frac12\leq b_\eta\leq1,\qquad
\xi_\eta\leq C_g\eta,\qquad
R_\eta-(\xi_\eta+b_\eta)\geq\kappa\eta.
\end{equation}
Consequently, for $Y\in B$ and $y=\gamma_\eta+b_\eta Y$,
\begin{equation}\label{eq:matrix-bounds}
|y-Y|\leq C_g\eta,\qquad
Q_{R_\eta}(y):=R_\eta^2I-y\otimes y\geq\kappa\eta I,
\end{equation}
after decreasing $\kappa$ if necessary.
\end{lemma}

\begin{proof}
From \eqref{eq:chosen-map},
\[
\frac{\xi_\eta}{\eta}\longrightarrow\frac{q}{1-\delta},
\qquad b_\eta\longrightarrow1.
\]
Moreover,
\[
\xi_\eta+b_\eta=1+(1-\delta)\xi_\eta,\qquad
\frac{R_\eta-(\xi_\eta+b_\eta)}{\eta}\longrightarrow c-q>0.
\]
This proves \eqref{eq:geom-bounds}.
The first estimate in \eqref{eq:matrix-bounds} follows from
$|y-Y|\leq \xi_\eta+|b_\eta-1|\leq(1+\delta)\xi_\eta$.
For the second, use
\[
Q_{R_\eta}(y)\geq(R_\eta^2-|y|^2)I
\geq\bigl(R_\eta^2-(\xi_\eta+b_\eta)^2\bigr)I
\]
and factor the last difference of squares.
\end{proof}

\begin{lemma}[Uniform transport estimate]\label{lem:transport}
There exists $C_T>0$, depending only on $N,p,c,q,\delta$, such that
for $0<\eta<\eta_g$ and
$(f,g)\in H^1(B_{R_\eta})\times L^2(B_{R_\eta})$,
\begin{equation}\label{eq:transport-X}
\|\mathcal T_{b_\eta,\gamma_\eta}(f,g)\|_{\X_1}
\leq C_T\|(f,g)\|_{\X_{R_\eta}}.
\end{equation}
\end{lemma}

\begin{proof}
We write $b,\gamma,R$ for the parameters in
\eqref{eq:chosen-map}. The image $\gamma+bB$ is contained in $B_R$
with the margin in Lemma~\ref{lem:geometry}.
At corresponding points $y=\gamma+bY$,
\[
\bigl\|(I-Y\otimes Y)-Q_R(y)\bigr\|_{\mathrm{op}}
\leq |1-R^2|+(|y|+|Y|)|y-Y|
\leq C\eta.
\]
The lower bound $Q_R(y)\geq\kappa\eta I$ therefore implies the
quadratic-form inequality
\begin{equation}\label{eq:form-compare}
I-Y\otimes Y\leq C Q_R(y).
\end{equation}
Since $\nabla_Y(b^kf(y))=b^{k+1}\nabla f(y)$,
\eqref{eq:form-compare} controls the degenerate gradient of the first
component of \eqref{eq:transport}.

For the second component, \eqref{eq:transport} and
Lemma~\ref{lem:geometry} give the pointwise estimate
\[
|\mathcal T_{b,\gamma}^{(2)}(f,g)(Y)|^2
\leq C\bigl(|g(y)|^2+|f(y)|^2+\eta^2|\nabla f(y)|^2\bigr).
\]
The final term is bounded by
$C\eta\,\nabla f(y)\cdot Q_R(y)\nabla f(y)$, hence by a constant
times the integrand in \eqref{eq:X}. The first component in $L^2$
is immediate. Integrating and using $dY=b^{-N}dy$,
with $b\geq1/2$, proves \eqref{eq:transport-X}.
\end{proof}

\begin{proof}[Proof of Theorem~\ref{thm:exclusion}]
Keep the choices \eqref{eq:choices}.
Reduce $\eta_*$ so that Lemmas~\ref{lem:geometry} and
\ref{lem:transport} apply, Proposition~\ref{prop:perturbative} applies
with the prescribed $m$ and $K=\{q\}$, and $1+c\eta_*<\delta_0^{-1}$.
Let $s$ be large, put $\tau=e^{-s}$ and $t=T_0-\tau$, and consider
the vertex
\begin{equation}\label{eq:admissible-vertex}
\widehat x=x_0-\xi_\eta\tau e,\qquad
\widehat T=T_0-\delta\xi_\eta\tau.
\end{equation}
For all sufficiently large $s$, uniformly for $0<\eta<\eta_*$,
this vertex has positive time and lies strictly inside the cone
\eqref{eq:nc}, since $\delta>\delta_0$.
It is therefore admissible. Moreover
\[
\widehat T-t=b_\eta\tau>0,\qquad
\frac{\widehat x-x_0}{\tau}=\gamma_\eta.
\]
The phase data at this vertex and the same physical time are exactly
$\mathcal T_{b_\eta,\gamma_\eta}(w,w_s)$.

The shift in \eqref{eq:nu} is
\begin{equation}\label{eq:positive-shift}
\nu=\frac{(\rho-\delta)\xi_\eta}{1-\delta\xi_\eta}
=q\eta.
\end{equation}
Thus the construction produces a positive shift beyond the limiting
sign threshold for every $\lambda\geq m$, while keeping the vertex admissible for the actual
solution. By \eqref{eq:exact-transport},
\[
\mathcal T_{b_\eta,\gamma_\eta}(W_d,0)
=\mathbf W_{d,q\eta}.
\]

Suppose that, for some $\lambda\geq m$, the left-hand side of
\eqref{eq:main-exclusion} is smaller than $\varepsilon_*$.
By linearity and Lemma~\ref{lem:transport}, the data in the new frame
satisfy
\[
\|(\widehat w,\widehat w_{\widehat s})
-\sigma\lambda\mathbf W_{d,q\eta}\|_{\X_1}
\leq C_T\varepsilon_*\lambda\eta^{\mu/2}.
\]
Choose $\varepsilon_*>0$ so that
$C_T\varepsilon_*<\epsilon_0$, with $\epsilon_0$ from
Proposition~\ref{prop:perturbative}.
It follows that
$H[\widehat w,\widehat w_{\widehat s}]<0$.
This contradicts \eqref{eq:nonnegative} at the admissible vertex
\eqref{eq:admissible-vertex}, and proves \eqref{eq:main-exclusion}.
\end{proof}

\begin{remark}\label{rem:collar}
The strict inequality $q<c$ supplies the margin
$R_\eta-(\xi_\eta+b_\eta)\gtrsim\eta$.
It is used both for the inclusion of the observation balls and for
the norm comparison \eqref{eq:form-compare}.
The sufficient relation $c>q_m$ is consequently explicit.
In particular the amplitude lower bound $m>0$ is fixed throughout;
uniformity as $m\downarrow0$ is not asserted.
\end{remark}

\subsection{The upper blow-up bound}\label{sec:bound}

With the concentration excluded, the profile parameters lie in a
compact subset of $\{|d|<1\}$. The sign of the energy then bounds
the amplitude, and the enlarged observation ball controls the
full gradient on $B$.

\begin{proof}[Proof of Theorem~\ref{thm:bound}]
Let $\eta_*,\varepsilon_*$ be as in
Theorem~\ref{thm:exclusion}, with $\eta_*<\eta_c$.
Choose $\varepsilon<\min\{\varepsilon_*,1\}$, to be decreased below.
At a sufficiently large time $s$, consider any comparison for which
\begin{equation}\label{eq:chosen-error}
\E^{R(\eta)}_{\lambda,d,\sigma}(w(s),w_s(s))<\varepsilon,
\qquad\eta=1-|d|,\quad\lambda\geq m.
\end{equation}
If $\eta<\eta_*$ then $R(\eta)=1+c\eta$, contrary to
Theorem~\ref{thm:exclusion}. Hence $\eta\geq\eta_*$.

We next bound the amplitude. For these parameters,
\begin{equation}\label{eq:denom-bound}
1-|d|R(\eta)
=\eta-c(1-\eta)\min\{\eta,\eta_c\}
\geq(1-c)\eta\geq(1-c)\eta_*.
\end{equation}
Thus the profiles and their first derivatives are uniformly bounded
on $B_{R(\eta)}$. In addition, compactness of
$\{|d|\leq1-\eta_*\}$ and positivity of $W_d$ give
\begin{equation}\label{eq:compact-potential}
P_d(0)\geq v_*>0,\qquad
\|(W_d,0)\|_{\X_1}+\|W_d^p\|_{L^{r'}(B)}\leq C_*.
\end{equation}
Changing the overall sign if needed, write
$(w,w_s)=\lambda(W_d,0)+(f,g)$.
Since restriction from $\X_R$ to $\X_1$ decreases the norm,
\[
\|(f,g)\|_{\X_1}\leq\varepsilon\lambda\eta^{\mu/2}
\leq\varepsilon\lambda.
\]
Lemma~\ref{lem:upper-perturbation} and \eqref{eq:compact-potential} give
\[
0\leq H[w,w_s]
\leq C_*(1+\varepsilon)^2\lambda^2
-(v_*-C_*\varepsilon)\lambda^{p+1},
\]
where $C_*$ may be enlarged. The left inequality holds at the original
admissible vertex $(x_0,T_0)$.
Reduce $\varepsilon$ so that $C_*\varepsilon\leq v_*/2$.
It follows that
\begin{equation}\label{eq:amplitude-bound}
\lambda^{p-1}\leq\frac{8C_*}{v_*}.
\end{equation}
This bound is deduced from the PDE sign, not imposed on the comparison.

Using \eqref{eq:denom-bound}, \eqref{eq:chosen-error}, and
\eqref{eq:amplitude-bound}, we obtain
\[
\|(w,w_s)\|_{\X_{R(\eta)}}
\leq\lambda\|(W_d,0)\|_{\X_{R(\eta)}}
+\varepsilon\lambda\eta^{\mu/2}\leq C.
\]
On the smaller ball $B$,
\begin{equation}\label{eq:coercive-on-B}
R(\eta)^2I-y\otimes y
\geq(R(\eta)^2-1)I
\geq2c\min\{\eta_*,\eta_c\}I.
\end{equation}
The last norm therefore controls the full $H^1(B)\times L^2(B)$
norm and gives $M(s)\leq C$.

Finally, if $\E_m(s)<\varepsilon$, the definition of the infimum
provides a comparison satisfying \eqref{eq:chosen-error}.
This proves \eqref{eq:pointwise-bound}, including the bounds for
every sufficiently accurate comparison.
Under \eqref{eq:approximation}, this gives a uniform bound on
$M(s)$ for all sufficiently large $s$.

It remains to express the bound in physical variables. Put
$\tau=T_0-t$ and $B_t=B(x_0,\tau)$. The change of variables gives
\[
\tau^{k-N/2}\|u(t)\|_{L^2(B_t)}=\|w(s)\|_{L^2(B)},\qquad
\tau^{k+1-N/2}\|\nabla u(t)\|_{L^2(B_t)}
=\|\nabla w(s)\|_{L^2(B)}.
\]
Also
\[
\tau^{k+1-N/2}\|u_t(t)\|_{L^2(B_t)}
=\|w_s+kw+y\cdot\nabla w\|_{L^2(B)}
\leq C_k M(s).
\]
The bound on $M(s)$ proves \eqref{eq:physical}, after increasing $C$.
\end{proof}

\section{Exclusion of an isolated dominant profile}\label{sec:multi}

Fix $m>0$ and $K\geq1$. A parameter $0<\eta_0<1/2$ will be
chosen below. Consider a finite family
\begin{equation}\label{eq:family}
d_j=\rho_j e_j,\quad |e_j|=1,\quad
\eta_j=1-\rho_j\in(0,\eta_0),\quad
\lambda_j>0,\quad \sigma_j=\pm1,\qquad 1\leq j\leq J.
\end{equation}
Select an index $i$ such that
\begin{equation}\label{eq:dominant}
E_j=\lambda_j^2\eta_j^\mu,\qquad E_i=\max_j E_j,
\qquad \lambda_i\geq m,\qquad
\sum_{j=1}^J\sqrt{E_j}\leq K\sqrt{E_i}.
\end{equation}
Put
\begin{equation}\label{eq:amplitude-comparison}
t_j=\sqrt{E_j/E_i},\qquad
\lambda_j=t_j\lambda_i(\eta_i/\eta_j)^{\mu/2},\qquad
0<t_j\leq1,\quad \sum_jt_j\leq K.
\end{equation}
When $J\leq J_0$, the last condition in \eqref{eq:dominant} holds
with $K=J_0$. It also allows families of unbounded cardinality,
provided the indicated sum remains controlled.

For $j\ne i$, define the separation scales
\begin{equation}\label{eq:omega}
\omega_{ij}=\max\left\{\eta_i,
\eta_i^{a/(2k)}\eta_j^{\beta/(2k)}\right\},
\qquad \beta=2k-a=k-\mu>0.
\end{equation}
Recall that $a+\beta=2k$. The geometric assumptions are
\begin{equation}\label{eq:isolation}
\eta_i\leq\varepsilon_g\min\{1,2(1-\delta_0)\},\qquad
\omega_{ij}\leq\varepsilon_g|e_i-e_j|^2\quad(j\ne i),
\end{equation}
where $\varepsilon_g>0$ will be chosen small. Only separation from
$i$ is required. Put
\begin{equation}\label{eq:theta}
\theta_i^2=\min\{1,2(1-\delta_0),\min_{j\ne i}|e_i-e_j|^2\},
\qquad \delta=1-\theta_i^2/4.
\end{equation}
The minimum over an empty set is $+\infty$. Thus $\theta_i>0$,
$\delta\geq(1+\delta_0)/2$, and
$\eta_i\leq\varepsilon_g\theta_i^2$.
For $q\in[1/2,1)$, set
\begin{equation}\label{eq:multi-map}
\xi_i=\frac{q\eta_i}{\rho_i-\delta+\delta q\eta_i},\qquad
b_i=1-\delta\xi_i,\qquad \gamma_i=-\xi_i e_i,\qquad
\nu_j=\frac{\xi_i(d_j\cdot e_i-\delta)}{b_i}.
\end{equation}
The observation region in the original coordinates is
$D_i=\gamma_i+b_iB$.

\begin{lemma}[A common admissible observation region]\label{lem:multi-geometry}
For a sufficiently small absolute $\varepsilon_g$, uniformly in
$q\in[1/2,1)$ and in all families satisfying \eqref{eq:isolation},
\begin{gather}
\tfrac12\leq b_i\leq1,\qquad
2\eta_i/\theta_i^2\leq\xi_i\leq8\eta_i/\theta_i^2,
\qquad D_i\subset B_{1+2\eta_i},\label{eq:multi-geometry}\\
\nu_i=q\eta_i,\qquad
\tfrac12\eta_i\leq-\nu_j\leq32\eta_i/\theta_i^2
\quad(j\ne i).\label{eq:other-shifts}
\end{gather}
Every $W_{d_j}$ is smooth on $\overline{D_i}$, and
\begin{equation}\label{eq:multi-transport}
\mathcal T_{b_i,\gamma_i}
\left(\sum_j\sigma_j\lambda_j(W_{d_j},0)\right)
=\sum_j\sigma_j\lambda_j\mathbf W_{d_j,\nu_j}.
\end{equation}
For all sufficiently large $s$, the vertex
\begin{equation}\label{eq:multi-vertex}
\widehat x_i=x_0-\xi_i e^{-s}e_i,\qquad
\widehat T_i=T_0-\delta\xi_i e^{-s}
\end{equation}
is admissible, uniformly over the families.
\end{lemma}

\begin{proof}
The denominator in \eqref{eq:multi-map} equals
$\theta_i^2/4-\eta_i(1-\delta q)$. For
$\varepsilon_g\leq1/8$ it lies between $\theta_i^2/8$ and
$\theta_i^2/4$. This gives the bounds for $\xi_i$.
Decreasing $\varepsilon_g$ so that $8\varepsilon_g\leq1/2$
gives the bound for $b_i$. Moreover,
\[
|\gamma_i|+b_i=1+(1-\delta)\xi_i\leq1+2\eta_i.
\]
For $j\ne i$, since $\theta_i\leq1$,
\[
d_j\cdot e_i\leq1-\theta_i^2/2,\qquad
\delta-d_j\cdot e_i\geq\theta_i^2/4.
\]
The first inequality follows separately when $e_j\cdot e_i$ is
nonnegative or negative. These estimates prove
\eqref{eq:other-shifts}; the identity for $\nu_i$ follows directly
from \eqref{eq:multi-map}.

At $y=\gamma_i+b_iY$,
\[
1+d_j\cdot y=b_i(1+d_j\cdot Y-\nu_j).
\]
The minimum of the second factor on $\overline B$ is
$(1-q)\eta_i>0$ for $j=i$, and $\eta_j+|\nu_j|>0$ otherwise.
Thus all profiles are smooth there, and \eqref{eq:exact-transport}
gives \eqref{eq:multi-transport}.
Finally $\delta-\delta_0\geq(1-\delta_0)/2$ and
$\xi_i\leq8\varepsilon_g$. The vertex \eqref{eq:multi-vertex}
lies strictly inside \eqref{eq:nc} and has positive time for
uniformly large $s$. Its remaining time at $t=T_0-e^{-s}$ is
$b_i e^{-s}>0$. This proves admissibility.
\end{proof}

\begin{lemma}[Uniform estimates for negative shifts]\label{lem:negative-shift}
For $\eta=1-|d|<1/2$, $v\geq0$, and
$r'=2(N+1)/(N+3)$,
\begin{align}
\|\mathbf W_{d,-v}\|_{\X_1}^2
&\leq C\eta^\mu,\label{eq:negative-X}\\
\|W_{d,-v}^p\|_{L^{r'}(B)}
&\leq C\eta^{(k+2)/2}
(\eta+v)^{-(\beta+2)/2}.
\label{eq:negative-dual}
\end{align}
The constants depend only on $N,p$.
\end{lemma}

\begin{proof}
The denominator is increased by $v$, so the mass and degenerate
gradient of $W_{d,-v}$ are bounded by those of $W_d$.
Moreover
\[
|Z_{d,-v}|=\frac{kv}{1+d\cdot y+v}W_{d,-v}
\leq kW_{d,-v}.
\]
Lemma~\ref{lem:profile-size} proves \eqref{eq:negative-X} for small
$\eta$. Enlarging the constant on the remaining compact parameter
range gives the claim for every $0<\eta<1/2$.
For any $\ell>a+1$, rotation and $z=1+e\cdot y$ give
\[
M_\ell(d,-v)
\leq C\int_0^2\frac{z^a}{(\eta+v+z/2)^\ell}\,dz
\leq C_\ell(\eta+v)^{a+1-\ell}.
\]
Apply this with $\ell=(k+2)r'$ as in Lemma~\ref{lem:dual}.
The identity $(a+1)/r'-(k+2)=-(\beta+2)/2$
gives \eqref{eq:negative-dual}.
\end{proof}

\begin{lemma}[Dominance on a cap]\label{lem:cap}
Let $\Gamma_i=\{Y\in B:0<1+e_i\cdot Y<\eta_i\}$ and $h=1-q$.
For $q\in[1/2,1)$ and $\eta_i\leq h$,
\begin{equation}\label{eq:positive-X-h}
\|\mathbf W_{d_i,q\eta_i}\|_{\X_1}^2
\leq C\eta_i^\mu h^{-\beta}.
\end{equation}
Without the restriction $\eta_i\leq h$,
\begin{equation}\label{eq:cap-potential}
\frac1{p+1}\int_{\Gamma_i}W_{d_i,q\eta_i}^{p+1}\,dY
\geq c\eta_i^\mu h^{-\beta-1}.
\end{equation}
The constants depend only on $N,p$. Under \eqref{eq:isolation},
if $\varepsilon_g$ is sufficiently small in terms of $N,p,K$,
\begin{equation}\label{eq:cap-dominance}
\left|\sum_j\sigma_j\lambda_jW_{d_j,\nu_j}(Y)\right|
\geq\tfrac12\lambda_iW_{d_i,q\eta_i}(Y)
\qquad(Y\in\Gamma_i).
\end{equation}
\end{lemma}

\begin{proof}
Rotate so that $e_i$ is the first coordinate vector.
Put $z=1+e_i\cdot Y$ and $L=\eta_i h$; the denominator is
$L+\rho_i z$, with $\rho_i\geq1/2$.
The mass is bounded using
$M_{2k}\leq2M_{2k+1}\leq C L^{-\beta}$.
For the degenerate gradient, use $1-(e_i\cdot Y)^2\leq2z$:
\[
c_{d_i}^2\int_0^2
\frac{z^{a+1}}{(L+\rho_i z)^{2k+2}}\,dz
\leq C\eta_i^k L^{-\beta}
=C\eta_i^\mu h^{-\beta}.
\]
Here $2k>a$ ensures convergence after extending the integral to
infinity. The squared velocity is at most
\[
C\eta_i^{k+2}L^{-\beta-1}
=C\eta_i^{\mu+1}h^{-\beta-1}.
\]
It is bounded by the preceding scale when $\eta_i\leq h$.
This proves \eqref{eq:positive-X-h}.

For the potential, substitute $z=\eta_i x$ on the cap.
The identity $c_{d_i}^{p+1}\asymp\eta_i^{k+1}$ and the
cross-sectional volume give
\[
\frac1{p+1}\int_{\Gamma_i}W_{d_i,q\eta_i}^{p+1}\,dY
\geq c\eta_i^\mu
\int_0^1\frac{x^a}{(h+x)^{2k+2}}\,dx.
\]
Restrict the last integral to $0<x<h$ and rescale $x=hX$.
This proves \eqref{eq:cap-potential}, uniformly for $0<h\leq1/2$.

For $j\ne i$, write $r_{ij}=|e_i-e_j|$ and
$Y=(-1+z)e_i+Y_\perp$. On the cap, $0<z<\eta_i$ and
$|Y_\perp|\leq\sqrt{2z}$. The component of $e_j$ orthogonal to
$e_i$ has length at most $r_{ij}$, so
\[
1+e_j\cdot Y\geq\frac{r_{ij}^2}{2}-z-r_{ij}\sqrt{2z}
\geq r_{ij}^2\left(\frac12-\varepsilon_g-
\sqrt{2\varepsilon_g}\right)\geq\frac{r_{ij}^2}{4}.
\]
Here we used $\eta_i\leq\varepsilon_g r_{ij}^2$ and reduced
$\varepsilon_g$. Since $\rho_j\geq1/2$ and $\nu_j<0$,
\[
W_{d_j,\nu_j}\leq C\eta_j^{k/2}r_{ij}^{-2k},
\qquad W_{d_i,q\eta_i}\geq c\eta_i^{-k/2}
\quad\hbox{on }\Gamma_i.
\]
Consequently, by \eqref{eq:amplitude-comparison},
\[
\frac{\lambda_jW_{d_j,\nu_j}}
{\lambda_iW_{d_i,q\eta_i}}
\leq C t_j\frac{\eta_i^{a/2}\eta_j^{\beta/2}}{r_{ij}^{2k}}
\leq C t_j\varepsilon_g^k.
\]
Summing and using $\sum_jt_j\leq K$ gives a bound
$CK\varepsilon_g^k$. Choose it at most $1/2$.
The triangle inequality proves \eqref{eq:cap-dominance}, independently
of the signs and of coincidences among the remaining directions.
\end{proof}

\begin{proposition}[A negative neighbourhood of the sum]
\label{prop:multi-negative}
There exist $q\in[1/2,1)$ and
$\eta_0,\varepsilon_g,\epsilon_*,b_*>0$, depending only on
$N,p,m,K,\delta_0$, such that the following holds.
For a family satisfying \eqref{eq:family}, \eqref{eq:dominant} and
\eqref{eq:isolation}, set
\[
\mathbf S=(S_1,S_2)=\sum_j\sigma_j\lambda_j\mathbf W_{d_j,\nu_j}.
\]
If $(v,z)\in H^1(B)\times L^2(B)$ and
\begin{equation}\label{eq:multi-small}
\|(v,z)-\mathbf S\|_{\X_1}
\leq\epsilon_*\lambda_i\eta_i^{\mu/2},
\end{equation}
then
\begin{equation}\label{eq:multi-negative}
H[v,z]\leq-b_*\lambda_i^{p+1}\eta_i^\mu<0.
\end{equation}
\end{proposition}

\begin{proof}
First choose $\varepsilon_g$ for the geometric and cap lemmas.
The triangle inequality for the norm induced by the positive
quadratic form $\mathcal Q$ gives
\[
\mathcal Q[\mathbf S]^{1/2}
\leq\sum_j\lambda_j\mathcal Q[\mathbf W_{d_j,\nu_j}]^{1/2}
\leq CK(1-q)^{-\beta/2}\sqrt{E_i},
\]
provided $\eta_i\leq1-q$. We used \eqref{eq:Q-equivalence},
\eqref{eq:negative-X}, \eqref{eq:positive-X-h}, and
\eqref{eq:dominant}. The cap gives potential at least
$c\lambda_i^{p+1}\eta_i^\mu(1-q)^{-\beta-1}$.
Therefore
\[
H[\mathbf S]\leq\lambda_i^{p+1}\eta_i^\mu(1-q)^{-\beta-1}
\left[CK^2m^{1-p}(1-q)-c\right].
\]
Choose $q$ close enough to one to make the bracket at most $-c/2$,
and then choose $\eta_0\leq1-q$ sufficiently small. Thus
\begin{equation}\label{eq:multi-model-negative}
H[\mathbf S]\leq-2b_*\lambda_i^{p+1}\eta_i^\mu,
\qquad \|\mathbf S\|_{\X_1}\leq C_{q,K}\lambda_i\eta_i^{\mu/2}.
\end{equation}

For the nonlinear first variation, Lemma~\ref{lem:dual} handles the
selected term. For $j\ne i$, the lower bound
$-\nu_j\geq\eta_i/2$ and Lemma~\ref{lem:negative-shift} give
\[
\lambda_j^p\|W_{d_j,\nu_j}^p\|_{r'}
\leq C t_j^p\lambda_i^p\eta_i^{\mu/2}
\frac{r_j^{p\beta/2}}{(1+r_j)^{(\beta+2)/2}},
\qquad r_j=\eta_j/\eta_i.
\]
The fraction is at most one, since $p\beta>0$ and
$p\beta-(\beta+2)=-\mu(p-1)<0$.
Taking $p$-th roots and applying the triangle inequality in
$L^{pr'}$ yields
\[
\|S_1\|_{pr'}\leq\sum_j\lambda_j\|W_{d_j,\nu_j}\|_{pr'}
\leq C_q\lambda_i\eta_i^{\mu/(2p)}\sum_j t_j.
\]
Hence, independently of $J$,
\begin{equation}\label{eq:multi-dual}
\||S_1|^{p-1}S_1\|_{r'}
\leq C_{q,K}\lambda_i^p\eta_i^{\mu/2}.
\end{equation}

Write $(f,g)=(v,z)-\mathbf S$ and $D=\|(f,g)\|_{\X_1}$.
Lemma~\ref{lem:upper-perturbation} gives
\[
H[v,z]\leq H[\mathbf S]+C_{q,K}
\left[(\lambda_i+\lambda_i^p)\eta_i^{\mu/2}D+D^2\right].
\]
By \eqref{eq:multi-small} and $\lambda_i\geq m$, the error is at
most $C_{m,q,K}(\epsilon_*+\epsilon_*^2)
\lambda_i^{p+1}\eta_i^\mu$. Choose $\epsilon_*$ small enough to
absorb it into \eqref{eq:multi-model-negative}.
\end{proof}

\begin{theorem}[Exclusion by an isolated dominant profile]\label{thm:multi}
Fix $m>0$, $K\geq1$ and $\delta_0\in(0,1)$, and take the constants
of Proposition~\ref{prop:multi-negative}. Every solution satisfying
\eqref{eq:nc} has the following property at every sufficiently large
$s$. For every family and selected index satisfying
\eqref{eq:family}, \eqref{eq:dominant} and \eqref{eq:isolation},
\begin{equation}\label{eq:unconditional-multi-gap}
\left\|\mathcal T_{b_i,\gamma_i}(w(s),w_s(s))
-\sum_j\sigma_j\lambda_j\mathbf W_{d_j,\nu_j}\right\|_{\X_1}
\geq\epsilon_*\lambda_i\eta_i^{\mu/2}.
\end{equation}
The constants are uniform in the cardinality, signs, amplitudes,
ratios of scales and angles, subject to the stated hypotheses.
Only the selected amplitude must exceed $m$.
The large-time threshold may depend on the solution and its cone.
\end{theorem}

\begin{proof}
Lemma~\ref{lem:multi-geometry} identifies the transported actual data
with the similarity data at the admissible vertex
\eqref{eq:multi-vertex}. They belong to $H^1(B)\times L^2(B)$.
If \eqref{eq:unconditional-multi-gap} failed,
Proposition~\ref{prop:multi-negative} would make their energy
strictly negative, contradicting \eqref{eq:nonnegative}.
All choices are uniform under the stated hypotheses.
\end{proof}

\begin{corollary}[Coalescing directions]\label{cor:coalescing}
Fix $J_0\geq1$, $m>0$ and $\delta_0\in(0,1)$.
There are constants as in Theorem~\ref{thm:multi} such that its
conclusion holds for $J\leq J_0$, any maximizing index with
$\lambda_i\geq m$, and every family with $\eta_j<\eta_0$ satisfying
\begin{equation}\label{eq:intro-angular}
\max_j\eta_j\leq\varepsilon_g
\min\left\{1,\,2(1-\delta_0),\,
\min_{j\ne\ell}|e_j-e_\ell|^2\right\},
\end{equation}
with the sufficiently small constant $\varepsilon_g$ of that theorem.
\end{corollary}

\begin{proof}
Take $K=J_0$. Since $a/(2k)+\beta/(2k)=1$,
$\omega_{ij}\leq\max_j\eta_j$. Thus \eqref{eq:intro-angular}
implies \eqref{eq:isolation}, and \eqref{eq:dominant} follows from
maximality of $E_i$.
\end{proof}

\begin{corollary}[A geometric alternative]\label{cor:alternative}
Use the constants of Theorem~\ref{thm:multi}. For sufficiently large
$s$, consider any family satisfying \eqref{eq:family},
\eqref{eq:dominant} and
$\eta_i\leq\varepsilon_g\min\{1,2(1-\delta_0)\}$.
Then at least one of the following holds: there is $j\ne i$ with
\begin{equation}\label{eq:cluster}
|e_i-e_j|^2<\varepsilon_g^{-1}
\max\{\eta_i,\eta_i^{a/(2k)}\eta_j^{\beta/(2k)}\};
\end{equation}
or the common observation region is admissible and the quantitative
separation \eqref{eq:unconditional-multi-gap} holds.
\end{corollary}

\begin{proof}
If \eqref{eq:cluster} fails for every $j\ne i$, then
\eqref{eq:isolation} holds. Apply Theorem~\ref{thm:multi}.
\end{proof}

For comparable $\eta_j$, the alternative \eqref{eq:cluster} singles
out angular distances of order $\sqrt\eta$. We next examine
interactions at that scale.

\section{Clusters at the concentration scale}\label{sec:clusters}

The number of profiles is fixed throughout this section. Their
radial concentrations are comparable; their directions may coincide
or differ at the scale of the concentrating caps.

\begin{lemma}[The limiting energy of a cluster]\label{lem:cluster-limit}
Let $\eta_n\downarrow0$, and fix a unit vector $e$. Suppose, for
$1\leq j\leq J$ with $J$ fixed, that
\begin{gather}
\frac{1-|d_{j,n}|}{\eta_n}\longrightarrow\alpha_j>0,
\qquad
\frac{e_{j,n}-e}{\sqrt{\eta_n}}\longrightarrow\zeta_j\in e^\perp,
\qquad e_{j,n}=\frac{d_{j,n}}{|d_{j,n}|},\label{eq:cluster-parameters}\\
\lambda_{j,n}\longrightarrow\lambda_j>0,\qquad
\frac{\nu_{j,n}}{\eta_n}\longrightarrow q_j<\alpha_j,
\qquad \sigma_j\in\{-1,1\}.
\label{eq:cluster-phases}
\end{gather}
Identify $e^\perp$ with $\R^{N-1}$, write
$Z=(X,V)\in\R^{N-1}\times\R^2$, and define
\begin{equation}\label{eq:cluster-function}
\Psi(Z)=\kappa_0\sum_j\sigma_j\lambda_j(2\alpha_j)^{k/2}
\left(\alpha_j-q_j+
\frac{|X+\zeta_j|^2+|V|^2}{2}\right)^{-k}.
\end{equation}
For $\mathbf S_n=\sum_j\sigma_j\lambda_{j,n}
\mathbf W_{d_{j,n},\nu_{j,n}}$,
\begin{equation}\label{eq:cluster-energy-limit}
\eta_n^{-\mu}H[\mathbf S_n]\longrightarrow
\mathcal H[\Psi]:=\frac1{2\pi}\int_{\R^{N+1}}
\left(\frac12|\nabla\Psi|^2-
\frac{|\Psi|^{p+1}}{p+1}\right)dZ.
\end{equation}
Both integrals on the right are finite. Moreover,
\begin{equation}\label{eq:cluster-perturbation-norms}
\|\mathbf S_n\|_{\X_1}
+\||S_{1,n}|^{p-1}S_{1,n}\|_{r'}
\leq C\eta_n^{\mu/2}.
\end{equation}
\end{lemma}

\begin{proof}
Write $\eta=\eta_n$ and use the coordinates
\[
y=-e+\sqrt\eta\,X+\eta z e,
\qquad |X|^2<2z-\eta z^2,\quad 0<z<2/\eta.
\]
The Jacobian is $\eta^{a+1}$. Since the directions have unit length,
\[
e_{j,n}\cdot e
=1-\tfrac12|e_{j,n}-e|^2
=1-\tfrac\eta2|\zeta_j|^2+o(\eta).
\]
Consequently, locally on
$\{(X,z):z\geq|X|^2/2\}$, the rescaled first component
$\eta^{k/2}S_{1,n}(y)$ converges, together with its first derivatives,
to
\[
\Phi(X,z)=\kappa_0\sum_j\sigma_j\lambda_j(2\alpha_j)^{k/2}
\left(z+\alpha_j-q_j+\frac{|\zeta_j|^2}{2}
+\zeta_j\cdot X\right)^{-k}.
\]
The limiting denominator is positive throughout this region.
After multiplication by the Jacobian and by $\eta^{-\mu}$, the
degenerate gradient density converges to
\[
|\nabla_X\Phi+X\partial_z\Phi|^2
+(2z-|X|^2)|\partial_z\Phi|^2,
\]
and the potential density converges to $|\Phi|^{p+1}$.
The mass terms tend to zero by \eqref{eq:moment-zero}; the squared
velocity terms are $O(\eta^{\mu+1})$ before normalization.
The mixed mass--velocity term tends to zero by Cauchy--Schwarz.
These estimates also apply to negative phases, since their
denominators are larger than at phase zero.

For completeness, the convergence has no unaccounted tail. Put
$z_0=1+e\cdot y$. For $z_0\geq R\eta$, with $R$ sufficiently large
in terms of the limiting parameters, the denominators of every
member are bounded below by $c z_0$, and
$1-(e_{j,n}\cdot y)^2\leq C z_0$.
The cross-sectional volume is at most $Cz_0^a$. Thus the normalized
gradient and potential tails of each member are bounded respectively
by
\[
C\eta^{k-\mu}\int_{R\eta}^2 z_0^{a-2k-1}\,dz_0
\leq CR^{-\beta},\qquad
C\eta^{k+1-\mu}\int_{R\eta}^2 z_0^{a-2k-2}\,dz_0
\leq CR^{-\beta-1}.
\]
The same bounds, with a different constant, hold for the finite sum.
Local convergence and then $R\to\infty$ justify the limiting
integrals.

Finally set $r=z-|X|^2/2$ and
$F(X,r)=\Phi(X,r+|X|^2/2)$. The limiting gradient density becomes
$|\nabla_XF|^2+2r|\partial_rF|^2$.
Writing $r=|V|^2/2$ in two radial variables gives
$dV=2\pi\,dr$ after angular integration. This proves
\eqref{eq:cluster-energy-limit}, with $F(X,|V|^2/2)=\Psi(X,V)$.
The tail bounds also prove finiteness.
Finally \eqref{eq:cluster-perturbation-norms} follows from
Lemmas~\ref{lem:profile-size}, \ref{lem:dual},
\ref{lem:negative-shift} and the triangle inequality in $L^{pr'}$.
\end{proof}

\begin{lemma}[A negative limit before the first singularity]
\label{lem:cluster-negative}
Fix the parameters $\alpha_j,\zeta_j,\lambda_j,\sigma_j$ above.
Let
\begin{equation}\label{eq:first-cluster}
D_* =\min_j\left(2\alpha_j+\frac{|\zeta_j|^2}{2}\right)>0.
\end{equation}
Group indices with the same pair $(\alpha_j,\zeta_j)$.
Assume that at least one group attaining this minimum has
\begin{equation}\label{eq:cluster-noncancel}
\sum_{j\text{ in that group}}\sigma_j\lambda_j\ne0.
\end{equation}
For $0<D<D_*$, put
$q_j(D)=D-\alpha_j-|\zeta_j|^2/2$ in
\eqref{eq:cluster-function}, and denote the result by $\Psi_D$.
Then
\begin{equation}\label{eq:cluster-negative-limit}
\mathcal H[\Psi_D]\longrightarrow-\infty
\qquad\text{as }D\uparrow D_*.
\end{equation}
In particular, \eqref{eq:cluster-noncancel} holds if the limiting
pairs $(\alpha_j,\zeta_j)$ are pairwise distinct.
\end{lemma}

\begin{proof}
Write $h=D_*-D$. Every denominator has the form
$h+c_j+(|X+\zeta_j|^2+|V|^2)/2$, with $c_j\geq0$.
A dilation by $\sqrt h$ gives
\[
\int_{\R^{N+1}}|\nabla\Psi_D|^2\,dZ\leq C h^{-\beta}
\qquad(0<h<1).
\]
Indeed the gradient integral of each member scales as
$h^{(N+1)/2-2k-1}=h^{-\beta}$, and the gradient of the finite sum
is controlled by the sum of these integrals.

Choose a minimum group satisfying \eqref{eq:cluster-noncancel},
and let $Z_0=(-\zeta_j,0)$ be its centre.
On $|Z-Z_0|\leq\sqrt h$, its combined contribution has absolute
value at least $c h^{-k}$, with $c>0$.
Every other group stays bounded there as $h\to0$: either its centre
is different, or its $c_j$ is strictly positive.
Thus $|\Psi_D|\geq c' h^{-k}$ on this ball for small $h$, and
\[
\int_{\R^{N+1}}|\Psi_D|^{p+1}\,dZ
\geq c''h^{(N+1)/2-k(p+1)}=c''h^{-\beta-1}.
\]
It follows that $\mathcal H[\Psi_D]\leq C h^{-\beta}
-c h^{-\beta-1}\to-\infty$.
\end{proof}

\begin{theorem}[Exclusion of a comparable cluster]\label{thm:cluster}
Let $u$ satisfy \eqref{eq:nc}, let $s_n\to\infty$, and consider
families satisfying \eqref{eq:cluster-parameters} and
$\lambda_{j,n}\to\lambda_j>0$, with fixed signs and fixed $J$.
Assume \eqref{eq:cluster-noncancel} for the minimum in
\eqref{eq:first-cluster}. Then there are $D\in(0,D_*)$ and
$\epsilon_c>0$ such that, for all sufficiently large $n$,
\begin{equation}\label{eq:cluster-gap}
\left\|\mathcal T_{b_n,-e/2}(w(s_n),w_s(s_n))
-\sum_j\sigma_j\lambda_{j,n}
\mathbf W_{d_{j,n},\nu_{j,n}}\right\|_{\X_1}
\geq\epsilon_c\eta_n^{\mu/2},
\end{equation}
where
\begin{equation}\label{eq:cluster-vertex}
b_n=\frac{1+D\eta_n}{2},\qquad
\nu_{j,n}=\frac{d_{j,n}\cdot e-1+D\eta_n}{1+D\eta_n}.
\end{equation}
All comparison profiles are smooth on the observation ball
$-e/2+b_nB$. The result is invariant under a common rotation of
each configuration.
\end{theorem}

\begin{proof}
Choose $D$ by Lemma~\ref{lem:cluster-negative} so that
$\mathcal H[\Psi_D]<0$. Formula \eqref{eq:cluster-parameters} gives
$\nu_{j,n}/\eta_n\to q_j(D)<\alpha_j$.
Lemma~\ref{lem:cluster-limit} therefore implies, for the sum
$\mathbf S_n$ in \eqref{eq:cluster-gap},
\[
H[\mathbf S_n]\leq-c\eta_n^\mu,
\quad \|\mathbf S_n\|_{\X_1}
+\||S_{1,n}|^{p-1}S_{1,n}\|_{r'}\leq C\eta_n^{\mu/2}.
\]
Lemma~\ref{lem:upper-perturbation} shows that perturbations of
$\X_1$ norm at most $\epsilon_c\eta_n^{\mu/2}$ still have negative
energy when $\epsilon_c$ is small enough.

Put $\tau_n=e^{-s_n}$ and $\delta_n=1-D\eta_n$.
For large $n$, $\delta_0<\delta_n<1$, so
\[
\widehat x_n=x_0-\tfrac12\tau_n e,\qquad
\widehat T_n=T_0-\tfrac12\delta_n\tau_n
\]
is an admissible vertex with remaining time $b_n\tau_n$.
The transport identity gives exactly \eqref{eq:cluster-vertex}.
Its profile denominators are positive since
$1-|d_{j,n}|-\nu_{j,n}>0$ for large $n$.
The actual transported data have nonnegative energy by
\eqref{eq:nonnegative}, proving \eqref{eq:cluster-gap}.
Rotation invariance gives the last assertion.
\end{proof}

\section{Growth estimates and sign constraints}\label{sec:comparison}

We compare the exclusions with two necessary conditions on blow-up
data: the growth estimates and the nonnegative sign of the
energy. The examples below are comparison data, not additional
blow-up solutions.

\subsection{Logarithmic growth estimates}

\begin{proposition}[Logarithmic bounds and degenerating profiles]
\label{prop:log-gap}
Fix $L>0$, $e\in\mathbb S^{N-1}$ and $c\in(q_\infty,1)$.
For large $s$, set
\[
\eta(s)=s^{-L},\qquad d(s)=(1-\eta(s))e,\qquad
F(y,s)=W_{d(s)}(y),\quad G(y,s)=\partial_sF(y,s).
\]
These are smooth on $B_{1+c\eta(s)}$. They satisfy
\begin{gather}
\|F(s)\|_{H^1(B)}^2+\|G(s)\|_2^2
\asymp s^{L(1-\mu)}\longrightarrow\infty,\label{eq:path-growth}\\
\E^{1+c\eta(s)}_{1,d(s),+}(F(s),G(s))\leq C/s,
\qquad H[F(s),G(s)]>0.\label{eq:path-error}
\end{gather}
Put $V=G+kF+y\cdot\nabla F$ and
\[
\mathcal J_F(s)=\frac12\int_B\left[
|\nabla F|^2-(y\cdot\nabla F)^2+|V|^2
-\frac{|F|^{p+1}}{p+1}\right]dy.
\]
For every fixed $\ell>0$,
\begin{align}
s^\ell\int_s^{s+\log2}e^{-2\mu\sigma}
\int_B(|\nabla F|^2+|V|^2)\,dy\,d\sigma&\longrightarrow0,
\label{eq:path-log1}\\
s^\ell e^{-2\mu s}|\mathcal J_F(s)|&\longrightarrow0,
\label{eq:path-log2}\\
s^\ell e^{-\frac{4k\mu}{2k+1}s}\|F(s)\|_2^2&\longrightarrow0.
\label{eq:path-log3}
\end{align}
The path cannot be the similarity evolution of a solution satisfying
\eqref{eq:nc} on these late-time collars.
\end{proposition}

\begin{proof}
On the indicated collar, $1+d(s)\cdot y\geq(1-c)\eta(s)$.
Differentiating \eqref{eq:Wd} with respect to $\eta$ gives
$|\partial_\eta W_{(1-\eta)e}|\leq C_c\eta^{-1}W_{(1-\eta)e}$.
Since $|\eta'|=L\eta/s$,
\[
|G|\leq C_{c,L}s^{-1}F,
\qquad \|G\|_{L^2(B_{1+c\eta})}\leq C s^{-1}\eta^{\mu/2},
\]
by Lemma~\ref{lem:profile-size}. This proves the error bound.
The same lemma gives \eqref{eq:path-growth}.
Proposition~\ref{prop:asymptotic} at $q=0$ gives
$H[F,0]\geq c_0\eta^\mu$ for small $\eta$.
The quadratic terms involving $G$ change this by at most
$C\eta^\mu/s$, so $H[F,G]>0$ for large $s$.

Moreover,
\[
\int_B\bigl(F^2+|\nabla F|^2+G^2+|V|^2+|F|^{p+1}\bigr)dy
\leq C(1+s^{L(1-\mu)}).
\]
Each polynomial factor is dominated by the decaying exponential
in \eqref{eq:path-log1}--\eqref{eq:path-log3}, proving those limits.
The change of variables in \eqref{eq:similarity} transforms
\eqref{eq:HZ} into the expression in \eqref{eq:path-log1}.
The expression in \cite[equation~(1.14)]{HZ25} becomes
$s^\ell e^{-2\mu s}\mathcal J_F(s)$, and its
Corollary~2 becomes the expression in \eqref{eq:path-log3}; here
\[
N-2k-\frac{(p-1)N}{p+3}=\frac{4k\mu}{2k+1}>0.
\]
Finally, \eqref{eq:path-error} contradicts Theorem~\ref{thm:exclusion}
with $m=1$
if this path were an actual evolution in a fixed non-characteristic
cone.
\end{proof}

The path satisfies the numerical growth estimates and the sign
condition at its original vertex. Its exclusion requires the sign
at a different admissible vertex, which uses the non-characteristic
evolution.

\subsection{The sign constraints at a single time}\label{sec:remaining}

One might try to recover the full energy norm by imposing the sign
of $H$ simultaneously in all contracted frames. The following
example shows that these inequalities do not control oscillation.

\begin{proposition}[Noncoercivity of the family of sign constraints]
\label{prop:noncoercivity}
There are smooth pairs $(f_n,g_n)$ on $\R^N$ such that
\begin{equation}\label{eq:all-signs}
H[\mathcal T_{b,\gamma}(f_n,g_n)]\geq0
\qquad\text{for every }0<b\leq1,\ \gamma\in\R^N,
\end{equation}
but
\begin{equation}\label{eq:unbounded-data}
\|f_n\|_{H^1(B)}+\|g_n\|_{L^2(B)}\longrightarrow\infty.
\end{equation}
These pairs are not asserted to be slices of solutions with a common
non-characteristic blow-up cone.
\end{proposition}

\begin{proof}
Let $\Lambda>0$ be defined in \eqref{eq:Q-square}, and choose $A>0$
so that $A^{p-1}\leq(p+1)\Lambda/2$.
Whenever $|v|\leq A$, \eqref{eq:Q-square} implies
\[
H[v,z]\geq\frac12\int_B
\left[(z-\mu v)^2+|\nabla v|^2-(y\cdot\nabla v)^2\right]dy
+\int_B\left(\frac{\Lambda}{2}
-\frac{A^{p-1}}{p+1}\right)v^2\,dy\geq0.
\]
Set $f_n(y)=A\sin(ny_1)$ and $g_n(y)=0$.
The first component after transport is
$b^kA\sin(n(\gamma_1+bY_1))$, whose absolute value is at most $A$
for $0<b\leq1$. Thus \eqref{eq:all-signs} holds regardless of the
transported velocity component.

Finally,
\[
\|\nabla f_n\|_{L^2(B)}^2
=A^2n^2\int_B\cos^2(ny_1)\,dy
=\frac{A^2|B|}{2}n^2+o(n^2),
\]
where the oscillatory integral tends to zero by the
Riemann--Lebesgue lemma applied to the cross-sectional volume of $B$.
This proves \eqref{eq:unbounded-data}.
\end{proof}

\section{Horizons}\label{sec:horizons}\label{sec:programme}

An upper bound uniform in similarity time would follow by excluding every sequence
$s_n\to\infty$ for which $M(s_n)\to\infty$. Theorem~\ref{thm:bound}
already gives the necessary condition
\begin{equation}\label{eq:bad-sequence}
s_n\longrightarrow\infty,\quad M(s_n)\longrightarrow\infty
\quad\Longrightarrow\quad
\liminf_{n\to\infty}\E_m(s_n)\geq\varepsilon.
\end{equation}
The next task is to derive information about such a sequence from
the equation. The following questions describe a possible
continuation of the argument; they are not additional conclusions
of the exclusion theorems.

\subsection{Profile extraction}

A useful extraction result must provide an approximation on the
region used by the change of vertex. For the isolated-profile
argument, this is $D_i=\gamma_i+b_iB$; for a comparable cluster, it
is $-e/2+b_nB$. Both regions may contain a collar outside $B$.
An expansion known only on $B$ would therefore leave part of the
argument uncontrolled.

The required remainder scale is equally restrictive. For a selected
profile it is $\lambda_i\eta_i^{\mu/2}$ in the transported
$\X_1$ norm. In contrast,
Lemma~\ref{lem:profile-size} gives
$\|\nabla W_d\|_2\asymp\eta^{(\mu-1)/2}$.
Smallness relative to the diverging full gradient norm is thus
insufficient. An extraction leading to
\eqref{eq:unconditional-multi-gap} must identify a dominant member
with the summability bound \eqref{eq:dominant}, while controlling
the remainder at the smaller degenerate-energy scale. Alternatively,
an extraction satisfying \eqref{eq:cluster-parameters} and
\eqref{eq:cluster-noncancel} could use
Theorem~\ref{thm:cluster}.

The profile decompositions of Bahouri and G\'erard \cite{BG99} and
the concentration--compactness and rigidity argument of Kenig and
Merle \cite{KM08} suggest a general strategy for separating
concentration from a remainder. Their energy-critical setting is
different: those results do not provide a decomposition into the family \eqref{eq:Wd}
in the present degenerate norm. In particular, an unbounded
sequence here would first require a normalization and a compactness
statement compatible with the boundary caps and the equation.

\subsection{Several scales and cancellation}

The limiting energy in Lemma~\ref{lem:cluster-limit} gives an explicit
object on which to study interacting caps. Its negative-potential
mechanism survives arbitrary signs as long as the first singular
group has a nonzero combined coefficient. Two cases remain outside
that argument: non-comparable radial scales and cancellation at
the first limiting singularity.

For non-comparable scales, one can seek a successive decomposition
into groups whose concentration ratios have finite positive limits.
Theorem~\ref{thm:multi} applies when a dominant profile separates
from the other directions; Lemma~\ref{lem:cluster-limit} gives the
candidate interaction energy within a comparable group. A proof
combining the two would need to estimate the effect of every
remaining group in the same admissible frame. Pairwise separation
alone does not establish that estimate when the observation region
changes from one scale to another.

If a minimizing group cancels, the leading limit loses information.
One must retain the differences between its finite-$n$ parameters
and determine the next nonzero term. It is not enough to delete
the cancelled limiting group and move the vertex beyond its first
singularity: the finite-$n$ comparison profiles still have their
own domains of definition. A refinement must keep their
denominators positive and control the remainder after transport.

\subsection{The role of the evolution}

Proposition~\ref{prop:noncoercivity} shows that a collection of
sign inequalities at one time cannot bound the full gradient.
An argument for arbitrary solutions must use the equation over
a time interval to control energy not captured by the explicit
profiles. One possible target is a localized alternative: either
the similarity energy remains bounded, or a sequence of times and
admissible regions carries a concentration of one of the types
excluded above, with the required remainder estimate.

The dissipation identity \eqref{eq:lyapunov} and the bounds of
\cite{HZ25} are natural inputs for choosing times and controlling
fluxes. Their exponential weight in similarity variables is a
specific difficulty: Proposition~\ref{prop:log-gap} exhibits
polynomially growing similarity data compatible with the numerical
conclusions. A successful argument must exploit the evolution and
the geometry to obtain additional control on these local scales.

\subsection{Nearly cancelling clusters}

The most direct extension of the cluster calculation is a
quantitative version of Lemma~\ref{lem:cluster-limit} when the
chosen vertex approaches the first singular group as $n$ grows.
The current proof fixes $D<D_*$ before letting $n\to\infty$.
Allowing $D=D_n\uparrow D_*$ would require estimates uniform in
the smallest profile denominator, in the differences between
nearly coincident parameters, and in the transported error.
Such estimates would test whether a nonzero group can still be
detected after leading cancellation, and whether a background
larger than the present error threshold can be allowed.

This is a concrete intermediate objective: obtain a quantitative
negative-energy criterion for clusters with nearly cancelling
parameters. Connecting that criterion to an extraction theorem
for actual superconformal solutions would then supply the next
step towards the universal bound.

\section*{Funding}

No funding was received for this work.

\section*{Declaration of competing interest}

The author declares no competing interests.

\section*{Data availability}

Data sharing is not applicable to this article as no datasets were generated or analysed during this study.


\begin{thebibliography}{99}
\small
\raggedright
\setlength{\itemsep}{0pt}

\bibitem{Alinhac95}
S. Alinhac,
\emph{Blowup for Nonlinear Hyperbolic Equations},
Progress in Nonlinear Differential Equations and their Applications,
vol.~17, Birkh\"auser, Boston, 1995.
\href{https://doi.org/10.1007/978-1-4612-2578-2}{doi:10.1007/978-1-4612-2578-2}.

\bibitem{AM01}
C. Antonini and F. Merle,
\emph{Optimal bounds on positive blow-up solutions for a semilinear wave equation},
Int. Math. Res. Not. (2001),
no.~21, 1141--1167.
\href{https://doi.org/10.1155/S107379280100054X}{doi:10.1155/S107379280100054X}.

\bibitem{BG99}
H. Bahouri and P. G\'erard,
\emph{High frequency approximation of solutions to critical nonlinear wave equations},
Amer. J. Math. \textbf{121} (1999), no.~1, 131--175.
\href{https://doi.org/10.1353/ajm.1999.0001}{doi:10.1353/ajm.1999.0001}.

\bibitem{CF86}
L. A. Caffarelli and A. Friedman,
\emph{The blow-up boundary for nonlinear wave equations},
Trans. Amer. Math. Soc. \textbf{297}
(1986), no.~1, 223--241.
\href{https://doi.org/10.1090/S0002-9947-1986-0849476-3}{doi:10.1090/S0002-9947-1986-0849476-3}.

\bibitem{CZ13}
R. C\^ote and H. Zaag,
\emph{Construction of a multisoliton blowup solution to the semilinear wave
equation in one space dimension},
Comm. Pure Appl. Math. \textbf{66}
(2013), no.~10, 1541--1581.
\href{https://doi.org/10.1002/cpa.21452}{doi:10.1002/cpa.21452}.

\bibitem{DS12}
R. Donninger and B. Sch\"orkhuber,
\emph{Stable self-similar blow up for energy subcritical wave equations},
Dyn. Partial Differ. Equ. \textbf{9}
(2012), no.~1, 63--87.
\href{https://arxiv.org/abs/1201.4337}{arXiv:1201.4337}.

\bibitem{HZ13}
M. A. Hamza and H. Zaag,
\emph{Blow-up results for semilinear wave equations in the super-conformal case},
Discrete Contin. Dyn. Syst. Ser. B \textbf{18}
(2013), no.~9, 2315--2329.
\href{https://doi.org/10.3934/dcdsb.2013.18.2315}{doi:10.3934/dcdsb.2013.18.2315}.

\bibitem{HZ25}
M. A. Hamza and H. Zaag,
\emph{A better bound on blow-up rate for the superconformal semilinear wave equation},
J. Funct. Anal. \textbf{288} (2025), no.~11, 110862.
\href{https://doi.org/10.1016/j.jfa.2025.110862}{doi:10.1016/j.jfa.2025.110862}.

\bibitem{KM08}
C. E. Kenig and F. Merle,
\emph{Global well-posedness, scattering and blow-up for the energy-critical
focusing non-linear wave equation},
Acta Math. \textbf{201} (2008), no.~2, 147--212.
\href{https://arxiv.org/abs/math/0610801}{arXiv:math/0610801}.

\bibitem{KL93a}
S. Kichenassamy and W. Littman,
\emph{Blow-up surfaces for nonlinear wave equations, I},
Comm. Partial Differential Equations \textbf{18}
(1993), nos.~3--4, 431--452.
\href{https://doi.org/10.1080/03605309308820936}{doi:10.1080/03605309308820936}.

\bibitem{KL93b}
S. Kichenassamy and W. Littman,
\emph{Blow-up surfaces for nonlinear wave equations, II},
Comm. Partial Differential Equations \textbf{18}
(1993), no.~11, 1869--1899.
\href{https://doi.org/10.1080/03605309308820997}{doi:10.1080/03605309308820997}.

\bibitem{KSV14}
R. Killip, B. Stovall and M. Vi\c{s}an,
\emph{Blowup behaviour for the nonlinear Klein--Gordon equation},
Math. Ann. \textbf{358} (2014), 289--350.
\href{https://doi.org/10.1007/s00208-013-0960-z}{doi:10.1007/s00208-013-0960-z}.

\bibitem{Levine74}
H. A. Levine,
\emph{Instability and nonexistence of global solutions to nonlinear wave
equations of the form $Pu_{tt}=-Au+\mathcal F(u)$},
Trans. Amer. Math. Soc. \textbf{192}
(1974), 1--21.
\href{https://doi.org/10.2307/1996814}{doi:10.2307/1996814}.

\bibitem{Liu26}
J. Liu,
\emph{A note on the stability of self-similar blow-up solutions for
superconformal semilinear wave equations},
J. Funct. Anal. \textbf{290} (2026), no.~8, 111364.
\href{https://doi.org/10.1016/j.jfa.2026.111364}{doi:10.1016/j.jfa.2026.111364}.

\bibitem{MZ03}
F. Merle and H. Zaag,
\emph{Determination of the blow-up rate for the semilinear wave equation},
Amer. J. Math. \textbf{125} (2003), no.~5, 1147--1164.

\bibitem{MZ05c}
F. Merle and H. Zaag,
\emph{Determination of the blow-up rate for a critical semilinear wave equation},
Math. Ann. \textbf{331} (2005), no.~2, 395--416.

\bibitem{MZ05g}
F. Merle and H. Zaag,
\emph{On growth rate near the blowup surface for semilinear wave equations},
Int. Math. Res. Not. (2005),
no.~19, 1127--1155.
\href{https://doi.org/10.1155/IMRN.2005.1127}{doi:10.1155/IMRN.2005.1127}.

\bibitem{MZ07}
F. Merle and H. Zaag,
\emph{Existence and universality of the blow-up profile for the semilinear
wave equation in one space dimension},
J. Funct. Anal. \textbf{253} (2007), no.~1, 43--121.
\href{https://doi.org/10.1016/j.jfa.2007.03.007}{doi:10.1016/j.jfa.2007.03.007}.

\bibitem{MZ08}
F. Merle and H. Zaag,
\emph{Openness of the set of non-characteristic points and regularity of the
blow-up curve for the 1 d semilinear wave equation},
Comm. Math. Phys. \textbf{282} (2008), 55--86.
\href{https://doi.org/10.1007/s00220-008-0532-3}{doi:10.1007/s00220-008-0532-3}.

\bibitem{MZ11r}
F. Merle and H. Zaag,
\emph{Blow-up behavior outside the origin for a semilinear wave equation
in the radial case},
Bull. Sci. Math. \textbf{135} (2011), no.~4, 353--373.
\href{https://doi.org/10.1016/j.bulsci.2011.03.001}{doi:10.1016/j.bulsci.2011.03.001}.

\bibitem{MZ12a}
F. Merle and H. Zaag,
\emph{Existence and classification of characteristic points at blow-up
for a semilinear wave equation in one space dimension},
Amer. J. Math. \textbf{134} (2012), no.~3, 581--648.
\href{https://arxiv.org/abs/0811.4068}{arXiv:0811.4068}.

\bibitem{MZ12b}
F. Merle and H. Zaag,
\emph{Isolatedness of characteristic points at blow-up for a
1-dimensional semilinear wave equation},
Duke Math. J. \textbf{161} (2012), no.~15, 2837--2908.
\href{https://doi.org/10.1215/00127094-1902040}{doi:10.1215/00127094-1902040}.

\bibitem{MZ15}
F. Merle and H. Zaag,
\emph{On the stability of the notion of non-characteristic point and
blow-up profile for semilinear wave equations},
Comm. Math. Phys. \textbf{333} (2015), no.~3, 1529--1562.
\href{https://doi.org/10.1007/s00220-014-2132-8}{doi:10.1007/s00220-014-2132-8}.

\bibitem{MZ16}
F. Merle and H. Zaag,
\emph{Dynamics near explicit stationary solutions in similarity variables
for solutions of a semilinear wave equation in higher dimensions},
Trans. Amer. Math. Soc. \textbf{368}
(2016), no.~1, 27--87.
\href{https://doi.org/10.1090/tran/6450}{doi:10.1090/tran/6450}.

\bibitem{MZpyramid}
F. Merle and H. Zaag,
\emph{Blowup solutions to the semilinear wave equation with a stylized
pyramid as a blowup surface},
Comm. Pure Appl. Math. \textbf{71}
(2018), no.~9, 1850--1937.
\href{https://doi.org/10.1002/cpa.21756}{doi:10.1002/cpa.21756}.

\bibitem{O24}
M. Ostermann,
\emph{Stable blowup for focusing semilinear wave equations in all dimensions},
Trans. Amer. Math. Soc. \textbf{377}
(2024), no.~7, 4727--4778.
\href{https://doi.org/10.1090/tran/9069}{doi:10.1090/tran/9069}.

\bibitem{Z26}
H. Zaag,
\emph{Soliton resolution at blow-up for the subconformal nonlinear wave equation},
in \emph{Proceedings of the International Congress of Mathematicians 2026},
vol.~5, Invited Lectures (Sections 9--11), SIAM, 2026, 368--387.
\href{https://doi.org/10.1137/25M1806892}{doi:10.1137/25M1806892}.

\end{thebibliography}
\end{document}